\documentclass[12pt,reqno]{amsart}
\usepackage{amssymb,amsmath,amsthm,amsfonts,latexsym}
\usepackage{mathrsfs}

\usepackage{geometry}
\newtheorem{theorem}{Theorem}[section]
\newtheorem{lemma}[theorem]{Lemma}

\numberwithin{equation}{section}

\def \C {{\mathbb C}}

\def \R {{\mathbb R}}

\def \Z {{\mathbb Z}}
\def \le{{\,\leqslant\,}}
\def \ge{{\,\geqslant\,}}

\begin{document}
\title[Representation by sums of unlike powers]{Representation by sums of unlike powers}
\author{Jianya Liu}
\address{School of Mathematics and Data Science Institute \\ Shandong University \\ Jinan  250100 \\ China}
\email{jyliu@sdu.edu.cn}
\author{Lilu Zhao}
\address{School of Mathematics \\ Shandong University \\ Jinan  250100 \\ China}
\email{zhaolilu@sdu.edu.cn}

\begin{abstract}It is proved that all sufficiently large integers $n$ can be represented as
$$n=x_1^2+x_2^3+\cdots+x_{13}^{14},$$
where $x_1,\ldots,x_{13}$ are positive integers. This improves upon the current record
with $14$ variables in place of $13$.
\end{abstract}

\maketitle


{\let\thefootnote\relax}

{\let\thefootnote\relax}

{\let\thefootnote\relax\footnotetext{This work is supported by the NSFC grants 12031008 and 11922113.}}

\section{Introduction}

This paper is concerned with the representation of natural numbers as sums of successive powers, starting with a square.
This problem was solved by Roth \cite{Roth} who proved that all sufficiently large positive integers $n$ can be expressed as
\begin{align}\label{rep}n=\sum_{i=1}^{s}x_i^{i+1}\end{align}
with $s=50$.  Subsequently, Roth's result was improved by  Thanigasalam \cite{Th1,Th2,Th3},
Vaughan \cite{Vaughan1970,Vaughan1971}, Br\"udern \cite{Br87,Br88} and Ford \cite{Ford1,Ford}. The current record is held by Ford \cite{Ford}, who proved in 1996 that all sufficiently large positive integers $n$ can be expressed in the form \eqref{rep} with $s=14$.

The main result in this paper is the following improvement.

\begin{theorem}\label{theorem1}All sufficiently large integers $n$ can be represented in the form
\begin{eqnarray}\label{13Thm1}
n=\sum_{i=1}^{13}x_i^{i+1},
\end{eqnarray}
where $x_1,\ldots,x_{13}$ are positive integers.
\end{theorem}

It is worth pointing out the new ingredients in our proof. The first one is the development of Davenport's iterative method to sums of unlike powers. Although Davenport's iterative method in the formulation of Vaughan (see Lemma 4 in \cite{Vaughan1986}) has been widely used in this topic, the method in this paper is quite different. We make use of unlike powers more effectively, and in particular we combine the diminishing range method with the mean value estimates of smooth Weyl sums. The details will be explained in Section 4.

The second new ingredient is the development of the large sieve inequality with an application to mean value estimates for smooth Weyl sums restricted to excessively large major arcs. Although the large sieve inequality is a well-known  method in number theory, it seems to be the first time that this method is introduced to deal with the sum of successive powers. We briefly explain it here. Let $Q\le \frac{1}{4}n^{1/2}$. Let $\mathfrak{n}(Q)=\mathfrak{M}(2Q)\setminus \mathfrak{M}(Q)$ with $\mathfrak{M}(Q)$ defined in \eqref{definefrakM}.
An important issue of this paper is to estimate
$$J_{k,s}(Q)=\int_{\mathfrak{n}(Q)}|F_2(\alpha)g_k(\alpha)^{2s}|d\alpha,$$
where $F_2(\alpha)$ is given in \eqref{defineFkalpha} and $g_k(\alpha)$ is a smooth Weyl sum defined in \eqref{definegk}. Of course one may conventionally use
$$J_{k,s}(Q)\le \Big(\sup_{\alpha \in \mathfrak{n}(Q)}|F_2(\alpha)|\Big) \int_{0}^1|g_k(\alpha)^{2s}|d\alpha$$
to deduce that
\begin{align}\label{J1}
J_{k,s}(Q)\ll n^{\frac{1}{2}}Q^{-\frac{1}{2}}n^{\frac{1}{k}\lambda_{k,s}}
\end{align}
where $\lambda_{k,s}$ is a permissible exponent, but this is not sufficient for our purpose. Instead,
we introduce the large sieve inequality to prove (up to an arbitrary small power of $n$) that
\begin{align}\label{J2}J_{k,s}(Q)\ll n^{\frac{1}{2}}Q^{-\frac{1}{2}}n^{-1+\frac{2s}{k}}Q^{\frac{2}{k}(k-2s+\lambda_{k,s})}.\end{align}
Since $2s-\lambda_{k,s}<k$ in the proof, the estimate \eqref{J2} improves upon \eqref{J1}, and the saving is crucial.
The method underlying the proof of \eqref{J2} is very flexible and it has other applications. We shall consider these applications elsewhere.

Besides, we make use of some important results in additive number theory. We apply an estimate for the seventh moment of the quartic smooth Weyl sum proved by Br\"udern and Wooley \cite{BW}. We refer the readers
to Wooley \cite{Wooley1995} for the method of breaking the classical convexity. We introduce a smooth function weight to the cubic Weyl sum and apply the mean value estimate for the fourth moment of a cubic Weyl sum proved by Br\"udern \cite{Br} using the method of the Kloosterman refinement. Therefore, our proof benefits from Weil's theory on the Riemann hypothesis over finite fields. We also follow the method of Ford \cite{Ford} to deal with $\int_0^1|f_4(\alpha)^2f_k(\alpha)^{2s}|d\alpha$.

As usual, we write $e(z)$ for $e^{2\pi iz}$. We assume
that $n$ is sufficiently large. We use $\ll$ and
$\gg$ to denote Vinogradov's well-known notations. The letter $\varepsilon$ denotes a sufficiently small positive real number. Any statement in which $\varepsilon$ occurs holds for each fixed $\varepsilon > 0$, and any implied constant in such a statement is allowed to depend on $\varepsilon$.

\vskip3mm

\section{Outline the proof}

Suppose that $2/3<\lambda<1$. We shall choose $\lambda$ in \eqref{definelambda}. Let
$$X_k=n^{\frac{1}{k}},\ \ Y_k=n^{\frac{\lambda}{k}}.$$
We introduce $w:\R\rightarrow \R_{\ge 0}$
\begin{align}\label{definew}w(t)=\exp\big(-\frac{1}{1/16-(t-3/4)^2}\big)\end{align}
and define
\begin{align}\label{defineFkalpha}F_k(\alpha)=\sum_{\substack{X_k/2\le x\le X_k}}w(x/X_k)e(x^k\alpha).\end{align}
We use $\mathcal{A}(P,R)$ to denote the set of $R$-smooth numbers up to $P$, that is
$$\mathcal{A}(P,R)=\{1\le x\le P:\ p|x,\ p \textrm{ prime } \Rightarrow p\le R\}.$$
We define
\begin{align*}f_k(\alpha)=\sum_{\substack{x\in \mathcal{A}(Y_k,R)}}e(x^k\alpha).\end{align*}

For convenience, we may choose $R=n^{\eta}$, where $\eta^{-1}$ is a sufficiently large positive integer (say $\eta^{-1}>10^{200}$) and $11!|\eta^{-1}$. Let
$$
r:= r_k=\frac{\eta^{-1}}{k}.
$$
For $k\le 11$, we have $r\in \Z^{+}$ and define
\begin{align}\label{definegk}g_k(\alpha)=\sum_{\substack{R/2<p_1,\cdots, p_{r}\le R}}e\big((p_1\cdots p_r)^k\alpha\big),\end{align}
where $p_1,\ldots,p_r$ are prime numbers.

Let $K_1=\{5,6,7,8,9,10,11\}$ and $K_2=\{4,12,13,14\}$. We define
\begin{align*}F(\alpha)=\prod_{k\in K_2}f_k(\alpha)\ \ \textrm{ and } \ G(\alpha)=\prod_{k\in K_1}g_k(\alpha).\end{align*}
We write
$$\mathcal{F}(\alpha)=F_2(\alpha)F_3(\alpha)F(\alpha)G(\alpha).$$
Now we introduce
$$\mathcal{N}(n)=\int_0^1\mathcal{F}(\alpha)e(-n\alpha)d\alpha.$$
Note that $\mathcal{N}(n)$ is the (weighted) number of solutions of \eqref{13Thm1},
and we shall finally prove $\mathcal{N}(n)\gg \mathcal{F}(0)n^{-1}$.

We introduce
\begin{align}\label{definefrakM}\mathfrak{M}(Q)=\bigcup_{q\le Q}\bigcup_{\substack{1\le a\le q \\ (a,q)=1}}\mathfrak{M}(q,a;Q),\end{align}
where
\begin{align*}\mathfrak{M}(q,a;Q)=\Big\{\alpha:\ |\alpha-\frac{a}{q}|\le\frac{Q}{qn}\Big\}.\end{align*}
Then for $Q\le \frac{1}{4}n^{1/2}$, we define
\begin{align*}\mathfrak{m}(Q)=\Big[\frac{1}{n^{1/2}},1+\frac{1}{n^{1/2}}\Big]\setminus\mathfrak{M}(Q).\end{align*}
We expect to prove
\begin{align*}\int_{\mathfrak{m}(Q)}|\mathcal{F}(\alpha)|d\alpha= o( \mathcal{F}(0)n^{-1})\end{align*}
with $Q$ as small as possible.

\begin{lemma}\label{lemmaboundG2}Let $\delta_1=0.0008985$. Then one has
\begin{align}\label{boundG2}\int_0^{1}|G(\alpha)|^2d\alpha\ll n^{-\frac{3}{4}-\delta_1+\varepsilon}G(0)^2.\end{align}\end{lemma}

Lemma \ref{lemmaboundG2} will be proved in Section 3 by applying the mean value estimates for smooth Weyl sums in Vaughan and Wooley \cite{VW,VW2}.

\begin{lemma}\label{lemmaboundF2}Let $\rho=0.004453$. Then one has
$$\int_0^{1}|F_3(\alpha)F(\alpha)|^2d\alpha\ll n^{-\frac{7}{9}+\rho}F_3(0)^2F(0)^2.$$\end{lemma}

We shall prove Lemma \ref{lemmaboundF2} by developing Davenport's iterative method in Section 4. Via a standard application of Lemmas \ref{lemmaboundG2}-\ref{lemmaboundF2} and Weyl's inequality, we can prove the following result.

\begin{lemma}\label{lemmaintmQ1}Let
$$Q_1=n^{\frac{1}{2}-\frac{1}{36}+\rho}.$$
 Then we have
$$\int_{\mathfrak{m}(Q_1)}|\mathcal{F}(\alpha)|d\alpha \ll \mathcal{F}(0)n^{-1-\frac{1}{2}\delta_1+\varepsilon}.$$
\end{lemma}

Note that $Q_1$ is very large. We shall develop the large sieve inequality to prove the following result.
\begin{lemma}\label{lemmaintmQ2}Let
\begin{align}\label{defineQ2}Q_2=n^{\frac{4}{9}+2\rho}.\end{align}
 Then we have
$$\int_{\mathfrak{m}(Q_2)\cap \mathfrak{M}(Q_1)}|\mathcal{F}(\alpha)|d\alpha \ll \mathcal{F}(0)n^{-1-\frac{4}{9}\delta_1}.$$
\end{lemma}
Note that $\frac{4}{9}+2\rho=\frac{1}{2}-2(\frac{1}{36}-\rho)<0.4533505$. This may be compared
with the work of Ford \cite{Ford} who dealt with the integration over the minor arcs $\mathfrak{m}(n^{\mu})$ with $\mu=0.461039$. Although $Q_2$ is (much) smaller than $Q_1$, it is still difficult to deal with the integration over $\mathfrak{M}(Q_2)$ by using the routine technique (see Theorem 4.1 \cite{V}).  We make use of an estimate for the seventh moment of the quartic smooth Weyl sum in \cite{BW} and an estimate for the fourth moment of a cubic Weyl sum in \cite{Br} to prove the following.
\begin{lemma}\label{lemmaintm29}One has
$$\int_{\mathfrak{m}(n^{2/9})\cap \mathfrak{M}(Q_2)}|\mathcal{F}(\alpha)|d\alpha \ll \mathcal{F}(0)n^{-1-\frac{2}{9}\delta_1}.$$
\end{lemma}

Now in order to prove Theorem \ref{theorem1}, it remains to deal with the integration on the (narrow) major arcs $\mathfrak{M}(n^{2/9})$. And the proof will be routine after establishing Lemma \ref{lemmaintm29}.

\vskip3mm

\section{Mean value estimates for smooth Weyl sums}

We define
\begin{align*}f_k(\alpha;N^{1/k},R)=\sum_{\substack{x\in \mathcal{A}(N^{1/k},R)}}e(x^k\alpha).\end{align*}
We say that an exponent $\lambda_{k,s}$ is permissible if it has the property that, for each $\varepsilon>0$, there exists a positive number $\eta=\eta(\varepsilon,k,s)$ such that whenever $R\le N^{\eta}$, one has
\begin{align}\label{meansmooth}\int_0^{1}|f_k(\alpha;N^{1/k},R)|^{2s}d\alpha \ll N^{\frac{1}{k}(\lambda_{k,s}+\varepsilon)}.\end{align}
Throughout this paper, we use $\lambda_{k,s}$ to denote permissible exponents. Since we only consider finitely many pairs of $k$ and $s$, we may say that for each $\varepsilon>0$, there exists a positive number $\eta=\eta(\varepsilon)$ such that if $R\le N^{\eta}$ then \eqref{meansmooth} holds. Furthermore, if we use $\lambda_{k,s}^\ast=\lambda_{k,s}+10^{-10}$ instead of $\lambda_{k,s}$, then there exists an absolute constant $\eta>0$ such that whenever $R\le N^{\eta}$, one has
\begin{align*}\int_0^{1}|f_k(\alpha;N^{1/k},R)|^{2s}d\alpha \ll N^{\frac{1}{k}\lambda_{k,s}^\ast}.\end{align*}

It is important in the proof that $\delta_1$ in \eqref{boundK1} and $\delta_2$ in \eqref{boundThetawithdelta2} are positive, and the values of $\delta_1$ and $\delta_2$ heavily depend on numerical values of permissible exponents. However, if we use $\lambda_{k,s}^\ast$ instead of $\lambda_{k,s}$ in the proof, then \eqref{boundK1} and \eqref{boundThetawithdelta2} will hold with $\delta_1$ and $\delta_2$ replaced by $\delta_1^\ast$ and $\delta_2^\ast$ respectively, where $\delta_j^\ast>\delta_{j}-10^{-9}$. In particular, both $\delta_1^\ast$ and $\delta_2^\ast$ are positive. Therefore, we shall not distinguish $\lambda_{k,s}$ and $\lambda_{k,s}^\ast$.

For permissible exponents $\lambda_{k,s}$, we introduce
\begin{align}\label{definealphaks}\alpha_{k,s}=\frac{2s-\lambda_{k,s}}{k},\end{align}
and therefore,
\begin{align}\label{boundfks-alpha}\int_0^{1}|f_k(\alpha;N^{1/k},R)|^{2s}d\alpha \ll N^{\frac{2s}{k}-\alpha_{k,s}}.\end{align}

We define
\begin{align*}\mathcal{K}_1=\int_0^{1}\prod_{k\in K_1}|f_{k}(\alpha;N^{1/k},R)|^2d\alpha.\end{align*}
Let
$$\kappa_1=\sum_{k\in K_1}\frac{2}{k}.$$
Then we have the trivial bound
$$\mathcal{K}_1\ll N^{\kappa_1}.$$

\begin{lemma}\label{lemmaboundK1}Let $\delta_1$ be given in Lemma \ref{lemmaboundG2}. Then one has
\begin{align}\label{boundK1}\mathcal{K}_1 \ll N^{\kappa_1-\frac{3}{4}-\delta_1}.\end{align}
\end{lemma}
\begin{proof}
We write
\begin{align*}U_{k,s}(N)=\int_0^{1}|f_k(\alpha;N^{1/k},R)|^{2s}d\alpha .\end{align*}By H\"older's inequality,
\begin{align}\label{holder1}\mathcal{K}_1\le \prod_{k\in K_1}U_{k,s_k}(N)^{\frac{1}{s_k}},\end{align}
where
\begin{align}\label{definesk1}s_5=4,\ s_6=6,\ s_8=8,\ s_9=9,\ s_{10}=10,\ s_{11}=11\end{align}
and $s_{7}$ is determined by
\begin{align}\label{definesk2}\sum_{k\in K_1}\frac{1}{s_k}=1.\end{align}
Note that $s_7=6.3974151$. We deduce by H\"older's inequality again that
\begin{align}&U_{7,s_7}(N)^{\frac{1}{s_7}}\ll U_{7,6}(N)^{\frac{7}{s_7}-1}
U_{7,7}(N)^{1-\frac{6}{s_7}}.\label{holder2}\end{align}

One has permissible exponents
\begin{align*}&\lambda_{5,4}=4.4386563,\ \ \ \ \ \lambda_{6,6}=7.2315633,\ \ \ \ \lambda_{7,6}=7.0143820,\
 \\ & \lambda_{7,7}=8.5410894,\ \ \ \ \ \lambda_{8,8}=9.8428621,\ \ \ \ \lambda_{9,9}=11.1425026,\
  \\ & \lambda_{10,10}=12.4375675,\ \lambda_{11,11}=13.7292224,\end{align*}
whence by \eqref{definealphaks} and \eqref{boundfks-alpha}, one has
\begin{align}\label{boundUks} U_{k,s}(N) \ll N^{\frac{2s}{k}-\alpha_{k,s}}\end{align}
with
\begin{align*}&\alpha_{5,4}=0.7122687,\ \ \ \ \ \alpha_{6,6}=0.7947394,\ \ \ \ \alpha_{7,6}=0.7122311,\
 \\ & \alpha_{7,7}=0.7798443,\ \ \ \ \ \alpha_{8,8}=0.7696422,\ \ \ \ \alpha_{9,9}=0.7619441,\
  \\ & \alpha_{10,10}=0.7562432,\ \ \ \alpha_{11,11}=0.7518888.\end{align*}
The values of $\lambda_{5,4}$ and $\lambda_{6,6}$ are in Appendix in \cite{VW}, and values of $\lambda_{k,s}$ for $7\le k\le 11$ can be found in Sections 9-13 in \cite{VW2}.

  We deduce from \eqref{holder1}, \eqref{holder2} and \eqref{boundUks} that
\begin{align*}\mathcal{K}_1\ll
N^{\kappa_1-\alpha(K_1)},\end{align*}
where
\begin{align}\label{definealphaK1}\alpha(K_1)=\sum_{k\in K_1\setminus \{7\}}\frac{\alpha_{k,s_k}}{s_k}+(\frac{7}{s_7}-1)\alpha_{7,6}+(1-\frac{6}{s_7})\alpha_{7,7}.\end{align}

Numerical computation yields
 \begin{align*}\alpha(K_1)= 0.7508985.\end{align*}
 This completes the proof.
\end{proof}

\noindent {\it Proof of Lemma \ref{lemmaboundG2}.}  By the definition of $g_k(\alpha)$ in \eqref{definegk}, we have
\begin{align}\label{boundg0}n^{1/k}(\log n)^{-\frac{1}{k\eta}}\ll g_k(0)\ll n^{1/k}(\log n)^{-\frac{1}{k\eta}}.\end{align}
For any $t\in \Z^{+}$, we define
\begin{align}\label{definetau}\tau_t(x)=\sum_{\substack{R/2<p_1,\cdots, p_{t}\le R \\ p_1\cdots p_t=x}}1.\end{align}
One has $\tau_t(x)\le t!$. We can express $g_k$ in the form
$$g_k(\alpha)=\sum_{x\le n^{1/k}}\tau_r(x)e(x^k\alpha).$$
On considering the solutions of the underlying diophantine equations, we can deduce from Lemma \ref{lemmaboundK1} that
\begin{align}\label{boundintG}\int_0^1|G(\alpha)|^2d\alpha\ll \int_0^{1}\prod_{k\in K_1}|f_k(\alpha;n^{1/k}, R)|^2d\alpha \ll n^{\kappa_1-\frac{3}{4}-\delta_1}.\end{align}
By \eqref{boundg0}, one has $n^{\kappa_1}\ll n^{\varepsilon}G(0)^2$. Therefore, \eqref{boundG2} follows from \eqref{boundintG}. The proof of Lemma \ref{lemmaboundG2} is complete.

\vskip3mm
\section{Davenport's iterative method to sums of unlike powers}

Let $\mathcal{T}$ denote the number of solutions of
\begin{align}\label{equationT}x_1^3-x_2^3=y_1^4-y_2^4+y_3^{12}-y_4^{12}+y_5^{13}-y_6^{13}+y_7^{14}-y_8^{14},\end{align}
where $X_3/2\le x_1,x_2 \le X_3$ and
\begin{align}\label{conditionyi}y_1,y_2\in \mathcal{A}(Y_{4},R),\ y_3,y_4\in \mathcal{A}(Y_{12},R),\ y_5,y_6\in \mathcal{A}(Y_{13},R),\ y_7,y_8\in \mathcal{A}(Y_{14},R).\end{align}
On writing
$$\kappa_0=\frac{2}{3}+2\lambda(\frac{1}{4}+\frac{1}{12}+\frac{1}{13}+\frac{1}{14}),$$
one has the trivial bound
$$\mathcal{T}\ll n^{\kappa_0}.$$

For $j\in \{1,2\}$, we define
\begin{align}\label{defineIj}\mathcal{I}_j=\int_0^1|f_4(\alpha)f_{12}(\alpha)f_{13}(\alpha)f_{14}(\alpha)^j|^2d\alpha.\end{align}
As a routine application of Davenport's iterative method (see Lemma 4 in \cite{Vaughan1986}), one may deduce that
\begin{align*}\mathcal{T}\ll X_3\,\mathcal{I}_1+(n^{\lambda-\frac{2}{3}+\varepsilon}\mathcal{I}_1)^{1/2}Y_4Y_{12}Y_{13}Y_{14}.\end{align*}
Then even subject to the best possible estimate $\mathcal{I}_1\ll Y_4Y_{12}Y_{13}Y_{14}$, one can only obtain
\begin{align}\label{T1}\mathcal{T}\ll n^{\kappa_0-\frac{7}{9}+\rho'},\end{align}
where
\begin{align*}\rho'=\frac{4}{9}-\lambda(\frac{1}{4}+\frac{1}{12}+\frac{1}{13}+\frac{1}{14}) \ \textrm{ with }\ \lambda=\frac{\frac{4}{3}}{1+\frac{1}{4}+\frac{1}{12}+\frac{1}{13}+\frac{1}{14}}.\end{align*}
Note that $\rho'=0.0109875$. However, the estimate \eqref{T1} is insufficient for our proof. The purpose of this section is to
prove $\mathcal{T}\ll n^{\kappa_0-\frac{7}{9}+0.004453}$ by developing Davenport's iterative method to sums of unlike powers.

Let
$$H=\frac{16}{3}n^{\lambda-\frac{2}{3}}.$$
We use $\mathcal{S}$ to denote the number of solutions of
$$h(3x^2+3hx+h^2)=y_1^4-y_2^4+y_3^{12}-y_4^{12}+y_5^{13}-y_6^{13}+y_7^{14}-y_8^{14},$$
where $1\le |h|\le H$, $X_3/2\le x,x+h \le X_3$ and $y_1,\ldots,y_8$ satisfy \eqref{conditionyi}.

\begin{lemma}\label{lemmaboundT1}One has
\begin{align*}\mathcal{T}\le X_3\,\mathcal{I}_1+\mathcal{S}.\end{align*}\end{lemma}
\begin{proof}Subject to the condition \eqref{conditionyi}, the left hand side of \eqref{equationT} is no more than $4n^{\lambda}$. We deduce that
$$|(x_1-x_2)(x_1^2+x_1x_2+x_2^2)|\le 4n^{\lambda},$$
and if $x_1,x_2\ge \frac{1}{2}n^{\frac{1}{3}}$ then
$$|x_1-x_2|\le \frac{16}{3}n^{\lambda-\frac{2}{3}}.$$
By changing variables $x_1-x_2=h, x_2=x$, we deduce that
$$x_1^3-x_2^3=h(3x^2+3hx+h^2),$$
and therefore,
$$\mathcal{T}\le X_3\,\mathcal{I}_1+\mathcal{S}.$$
We remark that $X_3\mathcal{I}_1$ is the contribution from solutions of \eqref{equationT} with $x_1-x_2=0$. This completes the proof.
\end{proof}

\begin{lemma}\label{lemmaboundS}One has
\begin{align}\label{boundS}\mathcal{S}\ll H\mathcal{I}_2+n^\varepsilon(H\mathcal{I}_2)^{1/2}Y_4Y_{12}Y_{13}.\end{align}
\end{lemma}
\begin{proof}We define $r(a)$ to be the number of representations of $a$ as
$$a=y_1^4+y_2^{12}+y_3^{13},$$
 where $y_1\in \mathcal{A}(Y_{4},R)$, $y_2\in \mathcal{A}(Y_{12},R)$ and $y_3\in \mathcal{A}(Y_{13},R)$. Then we have
$$\mathcal{S}=\sum_{\substack{h,x,a_1,a_2,z_1,z_2
\\ h(3x^2+3hx+h^2)=a_1-a_2+z_1^{14}-z_2^{14}}}r(a_1)r(a_2).$$
We define $\kappa(b)$ to be the number of representations of $b$ as
$$b=y_1^4+y_2^{12}+y_3^{13}+z_1^{14}-z_2^{14},$$
 where $y_1\in \mathcal{A}(Y_{4},R)$, $y_2\in \mathcal{A}(Y_{12},R)$, $y_3\in \mathcal{A}(Y_{13},R)$ and $z_1,z_2\in \mathcal{A}(Y_{14},R)$. Then we also have
$$\mathcal{S}=\sum_{\substack{h,x,a,b
\\ h(3x^2+3hx+h^2)=b-a}}r(a)\kappa(b).$$

By symmetry,
\begin{align}\label{fromStoS0}\mathcal{S}\le 2\mathcal{S}_0,\end{align}
where
$$\mathcal{S}_0=\sum_{\substack{h,x,a_1,a_2,z_1,z_2
\\ h(3x^2+3hx+h^2)=a_1-a_2+z_1^{14}-z_2^{14}
\\ r(a_2)\le r(a_1)}}r(a_1)r(a_2).$$

We define
$$\kappa_a(b)=\sum_{\substack{a',z_1,z_2\\ a'+z_1^{14}-z_2^{14}=b \\ r(a')\ge r(a)}}r(a').$$
Then one has
$$
\mathcal{S}_0=\sum_{\substack{h,x,a,b
\\ h(3x^2+3hx+h^2)=b-a
}}r(a)\kappa_a(b).
$$
Note that
$$\kappa(b)=\sum_{\substack{a',z_1,z_2\\ a'+z_1^{14}-z_2^{14}=b }}r(a').$$
If $\kappa_a(b)\not=0$, then
$$\kappa(b)\ge \kappa_a(b)\ge r(a).$$
Now we deduce that
$$\mathcal{S}_0=\sum_{\substack{h,x,a,b
\\ h(3x^2+3hx+h^2)=b-a \\ r(a)\le \kappa(b)
}}r(a)\kappa_a(b),$$
and therefore,
\begin{align}\label{boundS0}\mathcal{S}_0\le \sum_{\substack{h,x,a,b
\\ h(3x^2+3hx+h^2)=b-a \\ r(a)\le \kappa(b)
}}r(a)\kappa(b).\end{align}

We introduce
\begin{align}\label{definerhoh}\rho_{h}(b)=\sum_{\substack{x,a\\ h(3x^2+3hx+h^2)+a=b \\ r(a)\le \kappa(b)}}r(a),\end{align}
and deduce from \eqref{boundS0} that
$$\mathcal{S}_0\le \sum_{\substack{h,b
}}\kappa(b)\rho_{h}(b).$$
By Cauchy's inequality,
\begin{align}\label{boundS0square}\mathcal{S}_0^2\le \Big(\sum_{\substack{h,b
}}\kappa(b)^2\Big)\Big(\sum_{\substack{h,b
}}\rho_{h}(b)^2\Big).\end{align}

We observe
$$\sum_{\substack{b
}}\kappa(b)^2=\mathcal{I}_2,$$
and therefore,
\begin{align}\label{boundfirstterm}\sum_{\substack{h,b
}}\kappa(b)^2\ll H\mathcal{I}_2.\end{align}

On recalling the definition of $\rho_{h}(b)$ in \eqref{definerhoh}, we conclude that
\begin{align}\label{key}\sum_{\substack{h,b
}}\rho_{h}(b)^2=\sum_{\substack{h,x_1,x_2,a_1,a_2\\ h(3x_1^2+3hx_1+h^2)+a_1=h(3x_2^2+3hx_2+h^2)+a_2=b \\ r(a_1),r(a_2)\le \kappa(b)}}r(a_1)r(a_2).
\end{align}
In order to deal with the right hand side of \eqref{key}, we distinguish two cases $a_1=a_2$ or not. We first consider the contribution from $a_1=a_2$. Since $a_1=a_2$ implies $x_1=x_2$, we deduce that
$$\sum_{\substack{h,x_1,a_1,b\\ h(3x_1^2+3hx_1+h^2)+a_1=b \\ r(a_1)\le \kappa(b)}}r(a_1)^2\le \sum_{\substack{h,x_1,a_1,b\\ h(3x_1^2+3hx_1+h^2)+a_1=b \\ r(a_1)\le \kappa(b)}}r(a_1)\kappa(b)\le \mathcal{S}.$$
Next we consider the contribution from $a_1\not=a_2$. We deduce from
$$h(3x_1^2+3hx_1+h^2)+a_1=h(3x_2^2+3hx_2+h^2)+a_2$$
that
$$3h(x_1-x_2)(x_1+x_2+h)=a_2-a_1.$$
For fixed $a_1,a_2$ with $a_1\not=a_2$, there are at most $O(n^\varepsilon)$ possible choices of $h,x_1,x_2$, and then $b$ is determined by
$h,x_1$ and $a_1$. Therefore, the contribution from $a_1\not=a_2$ is at most $O(Y_4^2Y_{12}^2Y_{13}^2n^\varepsilon)$.
Now we conclude that
\begin{align}\label{key2}\sum_{\substack{h,b
}}\rho_{h}(b)^2\ll \mathcal{S}+Y_4^2Y_{12}^2Y_{13}^2n^\varepsilon.
\end{align}

From \eqref{boundS0square}, \eqref{boundfirstterm} and \eqref{key2}, we deduce that
$$\mathcal{S}_0^2\ll H\mathcal{I}_2(\mathcal{S}+Y_4^2Y_{12}^2Y_{13}^2n^\varepsilon),$$
and by \eqref{fromStoS0},
$$\mathcal{S}^2\ll H\mathcal{I}_2(\mathcal{S}+Y_4^2Y_{12}^2Y_{13}^2n^\varepsilon).$$
Now \eqref{boundS} follows immediately from above, and the proof of the lemma is complete.\end{proof}

\begin{lemma}\label{lemmaboundT2}One has
\begin{align*}\mathcal{T}\ll X_3\mathcal{I}_1+H\mathcal{I}_2+n^\varepsilon(H\mathcal{I}_2)^{1/2}Y_4Y_{12}Y_{13}.\end{align*}\end{lemma}
\begin{proof}This follows from Lemma \ref{lemmaboundT1} and Lemma \ref{lemmaboundS}.\end{proof}
In order to deal with $\mathcal{I}_2$, we follow the approach developed by Ford \cite{Ford} (see (3.2) in \cite{Ford} and also Lemma 2.2 of Wooley \cite{Wooley1992}). Let
$$S_{k,s}=\int_0^1|f_4(\alpha)^2f_k(\alpha)^{2s}|d\alpha.$$
Note that $S_{k,s}$ is the same as $S_{k,s}^{(4)}(Y_4)$ in \cite{Ford}. For a pair of integers $(k,s)\in \{(13,4),(14,4),(14,5)\}$, we introduce
\begin{align}\label{definethetaks}\theta_{k,s}=\frac{\frac{4}{k}(\lambda_{k,2s}-2\lambda_{k,s})}{k+1+\lambda_{k,2s}-2\lambda_{k,s}}\end{align}
and
\begin{align}\label{definesigmaks}\sigma_{k,s}=\frac{1}{4}+\frac{s\theta_{k,s}}{2}+(\frac{1}{k}-\frac{\theta_{k,s}}{4})\lambda_{k,s}.\end{align}

\begin{lemma}Let $(k,s)\in \{(13,4),(14,4),(14,5)\}$. Let $\theta_{k,s}$ be given in \eqref{definethetaks} and $\sigma_{k,s}$ be given in \eqref{definesigmaks}. Then one has
\begin{align}\label{boundSks}S_{k,s}\ll Y_4^{\frac{4}{k}+\theta_{k,s}+\varepsilon}S_{k,s-1}+Y_4^{4\sigma_{k,s}+\varepsilon}R.\end{align}
\end{lemma}
\begin{proof}Let
$$\mathfrak{f}(\alpha)=\sum_{Y_4^{\theta}<m\le Y_4^{\theta}R} \sum_{1\le d\le Y_4^{1-k\theta}}\sum_{z\le 2Y_4}e\Big(\alpha \frac{(z+dm^k)^4-(z-dm^k)^4}{m^k}\Big).
$$
Note that $\mathfrak{f}(\alpha)$ coincides with the function $F_4(\alpha)$ in \cite{Ford} (see (3.4) in \cite{Ford}). One can deduce that (by Lemma 3.1 in \cite{Ford}, for example)
$$\int_0^{1}|\mathfrak{f}(\alpha)|^2d\alpha \ll Y_4^{2-(k-1)\theta +\varepsilon}R.$$
For $0<\theta<1/k$, by (3.7) in \cite{Ford} (with $h=4$, $P=Y_4$, $M=Y_4^{\theta}$, $Q=Y_4^{4/k-\theta}$ and $a=\frac{1}{2}$),  we have
\begin{align*}S_{k,s}\ll\, & Y_4^{\frac{4}{k}+\theta+\varepsilon}S_{k,s-1}
\\ &\ \ \ +Y_4^{(2s-1)\theta+\varepsilon}\Big\{Y_4^{1+\theta+(\frac{4}{k}-\theta)\lambda_{k,s}}R+
\Big(\int_0^{1}|\mathfrak{f}(\alpha)|^2d\alpha\Big)^{\frac{1}{2}} Y_4^{\frac{1}{2}(\frac{4}{k}-\theta)\lambda_{k,2s}}\Big\}.\end{align*}
Then we conclude that
\begin{align}\label{boundSks1}S_{k,s}\ll Y_4^{\frac{4}{k}+\theta+\varepsilon}S_{k,s-1}+Y_4^{(2s-1)\theta+\varepsilon}R\big(Y_4^{1+\theta+(\frac{4}{k}-\theta)\lambda_{k,s}}+
Y_4^{1-\frac{k-1}{2}\theta+\frac{1}{2}(\frac{4}{k}-\theta)\lambda_{k,2s}}\big).\end{align}
Note that $\theta_{k,s}$ in \eqref{definethetaks} is actually determined by
$$1+\theta+(\frac{4}{k}-\theta)\lambda_{k,s}=1-\frac{k-1}{2}\theta+\frac{1}{2}(\frac{4}{k}-\theta)\lambda_{k,2s},$$
and it follows from \eqref{boundSks1} that
\begin{align*}S_{k,s}\ll Y_4^{\frac{4}{k}+\theta_{k,s}+\varepsilon}S_{k,s-1}+Y_4^{4\sigma_{k,s}+\varepsilon}R.\end{align*}

We provide some numerical values in the following. One has permissible exponents (see Tables in Sections 15-16 in \cite{VW2})
\begin{align*}&\lambda_{13,4}=4.0980713,\ \ \lambda_{14,4}=4.0856057,\ \ \lambda_{14,5}=5.2216967,
\\ &\lambda_{13,8}=9.0257224,\ \ \lambda_{14,8}=8.9350975,\ \ \lambda_{14,10}=11.6442024.\end{align*}
Then by \eqref{definethetaks} and \eqref{definesigmaks}, we have
\begin{align*}\theta_{13,4}=0.01721257,\ \ \theta_{14,4}=0.01384513,\ \ \theta_{14,5}=0.02117723\end{align*}
and
\begin{align}\label{valuesigma}\sigma_{13,4}=0.58202682,\ \ \sigma_{14,4}=0.5553779,\ \ \sigma_{14,5}=0.6482762.\end{align}
This completes the proof.
\end{proof}
We remark that in our applications, the second term on the right hand side of \eqref{boundSks} will dominate the first.

\begin{lemma}Let $k\ge 12$. One has
\begin{align}\label{boundSk3}S_{k,3}\ll Y_4^{1+\varepsilon}Y_k^{\lambda_{k,3}}.\end{align}
\end{lemma}
\begin{proof}Note that $S_{k,3}$ is the number of solutions of
\begin{align}\label{equation4k}x_1^4-x_2^4=\sum_{j=1}^3(y_j^k-z_j^k),\end{align}
where $x_1,x_2\in \mathcal{A}(Y_4,R)$ and $y_1,y_2,y_{3},z_1,z_2,z_3\in \mathcal{A}(Y_k,R)$.

The number of solutions of \eqref{equation4k} with $x_1\not=x_2$ is $O(Y_k^{6+\varepsilon})$, since for any fixed $y_1,y_2,y_{3},z_1,z_2,z_3$ with
$\sum_{j=1}^3(y_j^k-z_j^k)\not=0$ there are at most $O(Y_k^{\varepsilon})$ possible choices of $x_1$ and $x_2$. The number of solutions of \eqref{equation4k} with $x_1=x_2$ is $O(Y_4S')$, where $S'$ denotes the number of solutions of
\begin{align*}\sum_{j=1}^3(y_j^k-z_j^k)=0\end{align*}
with $y_1,y_2,y_{3},z_1,z_2,z_3\in \mathcal{A}(Y_k,R)$. Note that $S'\ll Y_{k}^{\lambda_{k,3}}$. We conclude that
\begin{align*}S_{k,3}\ll Y_k^{6+\varepsilon}+Y_4Y_k^{\lambda_{k,3}}.\end{align*}
For $k\ge 12$, one has $Y_k^{6}\le Y_{4}Y_k^{3}$, and by $\lambda_{k,3}\ge 3$, we finally obtain
\begin{align*}S_{k,3}\ll Y_4^{1+\varepsilon}Y_k^{\lambda_{k,3}}.\end{align*}
This completes the proof.
\end{proof}

One has from Tables in Sections 14-16 in \cite{VW2} that
\begin{align*}\lambda_{12,3}=3.0173811,\ \ \lambda_{13,3}=3.0139128,\ \ \lambda_{14,3}=3.0113494.\end{align*}

\begin{lemma}\label{lemmaboundSks}Let $(k,s)\in \{(13,4),(14,4),(14,5)\}$. Then one has
\begin{align}\label{boundSks2}S_{k,s}\ll Y_4^{4\sigma_{k,s}+\varepsilon}R.\end{align}
\end{lemma}
\begin{proof}For $k\in \{13,14\}$, we conclude from \eqref{boundSks} and \eqref{boundSk3} that
\begin{align*}S_{k,4}\ll Y_4^{4\sigma_{k,4}'+\varepsilon}
+Y_4^{4\sigma_{k,4}+\varepsilon}R,
\end{align*}
where
$$\sigma_{k,4}'=\frac{1}{k}+\frac{1}{4}\theta_{k,4}+\frac{1}{4}+\frac{1}{k}\lambda_{k,3}.$$
Note that $\sigma'_{13,4}=0.5630657$ and $\sigma'_{14,4}=0.5399863$. On recalling \eqref{valuesigma}, we have
$\sigma'_{13,4}<\sigma_{13,4}$ and $\sigma'_{14,4}<\sigma_{14,4}$. Therefore, one has
\begin{align*}S_{k,4}\ll Y_4^{4\sigma_{k,4}+\varepsilon}R \ \textrm{ for } \ k\in \{13,14\}.
\end{align*}

We apply \eqref{boundSks} again to deduce that
\begin{align*}S_{14,5}\ll \Big(Y_4^{4\sigma_{14,5}'+\varepsilon}
+Y_4^{4\sigma_{14,5}+\varepsilon}\Big)R,
\end{align*}
where
$$\sigma_{14,5}'=\frac{1}{14}+\frac{1}{4}\theta_{14,5}+\sigma_{14,4}=0.6321008.$$
By \eqref{valuesigma}, we have $\sigma_{14,5}'<\sigma_{14,5}$. The proof of the lemma is complete.
\end{proof}

We introduce
$$u_1=\frac{1}{4}+\sum_{k=12}^{14}\frac{1}{3k}\lambda_{k,3}.$$
\begin{lemma}One has
\begin{align}\label{boundI1}\mathcal{I}_1\ll Y_4^{4u_1+\varepsilon}.\end{align}
\end{lemma}
\begin{proof}By H\"older's inequality, one has
$$\mathcal{I}_1\le S_{12,3}^{1/3}S_{13,3}^{1/3}S_{14,3}^{1/3}.$$
Then we deduce from \eqref{boundSk3} that
\begin{align*}\mathcal{I}_1\ll Y_4^{1+\varepsilon}Y_{12}^{\frac{1}{3}\lambda_{12,3}}
Y_{13}^{\frac{1}{3}\lambda_{13,3}}Y_{14}^{\frac{1}{3}\lambda_{14,3}}.\end{align*}
This completes the proof of \eqref{boundI1} since $Y_k=Y_4^{4/k}$.
\end{proof}

We introduce
$$u_2=\frac{1}{3}(\frac{1}{4}+\frac{1}{12}\lambda_{12,3})+\frac{1}{4}\sigma_{13,4}+\frac{1}{12}\sigma_{14,4}+\frac{1}{3}\sigma_{14,5}.$$
\begin{lemma}One has
\begin{align}\label{boundI2}\mathcal{I}_2\ll Y_4^{4u_2+\varepsilon}R.\end{align}
\end{lemma}
\begin{proof}By H\"older's inequality, one has
$$\mathcal{I}_2\le S_{12,3}^{1/3}S_{13,4}^{1/4}S_{14,4}^{1/12}S_{14,5}^{1/3}.$$
Then \eqref{boundI2} follows from \eqref{boundSk3} and \eqref{boundSks2}.
\end{proof}

Now we are able to establish the upper bound of $\mathcal{T}$.
\begin{lemma}\label{lemmaboundT}Let $\rho=0.004453$. Then one has
$$\mathcal{T}\ll n^{\kappa_0-7/9+\rho}.$$\end{lemma}
\begin{proof}
We deduce from Lemma \ref{lemmaboundT2}, \eqref{boundI1} and \eqref{boundI2} that
$$\mathcal{T}\ll n^{\varepsilon}R(n^{\frac{1}{3}+\lambda u_1}
+n^{\lambda-\frac{2}{3}+\lambda u_2}+n^{\frac{1}{2}\lambda-\frac{1}{3}+\frac{1}{2}\lambda u_2+\frac{1}{4}\lambda+\frac{1}{12}\lambda+\frac{1}{13}\lambda}).
$$
We choose $\lambda$ by equating $\frac{1}{3}+\lambda u_1$ and $\lambda-\frac{2}{3}+\lambda u_2$, that is
\begin{align}\label{definelambda}\lambda=\frac{1}{1+u_2-u_1}=0.9155538.\end{align}
Note that $u_1=0.4827948$, $u_2=0.5750298$ and $\frac{1}{3}+\lambda u_1-(\kappa_0-7/9)=0.004453$. The proof is complete.
\end{proof}

We remark that Lemma \ref{lemmaboundF2} follows from Lemma \ref{lemmaboundT} immediately.
Now we prove Lemma \ref{lemmaintmQ1} by using Lemma \ref{lemmaboundG2} and Lemma \ref{lemmaboundF2}.

\vskip3mm

\noindent {\it Proof of Lemma \ref{lemmaintmQ1}.}  By Schwarz's inequality,
\begin{align*}\int_{0}^1|F_3(\alpha)F(\alpha)G(\alpha)|d\alpha\le
\Big(\int_0^1|F_3(\alpha)F(\alpha)|^2d\alpha\Big)^{1/2} \Big(\int_0^1|G(\alpha)|^2d\alpha\Big)^{1/2}.\end{align*}
Then it follows from Lemma \ref{lemmaboundG2} and Lemma \ref{lemmaboundF2} that
\begin{align}\label{boundF3FG}\int_{0}^1|F_3(\alpha)F(\alpha)G(\alpha)|d\alpha \ll F_3(0)F(0)G(0)n^{-\frac{3}{4}-\frac{1}{72}+\frac{1}{2}\rho-\frac{1}{2}\delta_1+\varepsilon}.\end{align}

By Weyl's inequality (Lemma 2.4 in \cite{V}) and the partial summation formula, one has
\begin{align}\label{applyWeyl}\sup_{\alpha \in \mathfrak{m}(Q)}|F_2(\alpha)|\ll F_2(0)^{1+\varepsilon}Q^{-\frac{1}{2}}.\end{align}
Now we conclude from \eqref{boundF3FG} and \eqref{applyWeyl} that
\begin{align*}\int_{\mathfrak{m}(Q)}|\mathcal{F}(\alpha)|d\alpha \ll \mathcal{F}(0)Q^{-\frac{1}{2}}n^{-\frac{3}{4}-\frac{1}{72}+\frac{1}{2}\rho-\frac{1}{2}\delta_1+\varepsilon}.\end{align*}
On choosing $Q=Q_1$, we obtain
$$\int_{\mathfrak{m}(Q_1)}|\mathcal{F}(\alpha)|d\alpha \ll \mathcal{F}(0)n^{-1-\frac{1}{2}\delta_1+\varepsilon}.$$
This completes the proof of Lemma \ref{lemmaintmQ1}.

\vskip3mm

\section{Large sieve inequality}

Let
$$\nu=10^{-100}.$$
We define
\begin{align*}\mathfrak{M}^\ast(Q)=\bigcup_{q\le Q}\bigcup_{\substack{1\le a\le q \\ (a,q)=1}}\mathfrak{M}^\ast(q,a;Q),\end{align*}
where
\begin{align*}\mathfrak{M}^\ast(q,a;Q)=\Big\{\alpha:\ |\alpha-\frac{a}{q}|\le \min(\frac{Q}{qn}, n^{\nu-1})\Big\}.\end{align*}
In this section, we consider
\begin{align*}\mathcal{J}_{k,s}(Q)=\int_{\mathfrak{M}^\ast(Q)}|g_k(\alpha)|^{2s}d\alpha.\end{align*}
For the proof of Theorem \ref{theorem1}, we only need to consider $\mathcal{J}_{k,s}(Q)$ for $5\le k\le 11$. The results in this section hold for all $k\in \Z^{+}$ providing that $r=\eta^{-1}/k$ is sufficiently large.

The following result is well-known. One may refer to Lemma 5.3 in \cite{V}.
\begin{lemma}[Large sieve inequality]\label{largesieve1}Let $\delta>0$. Suppose that $\Gamma$ is a set of $\delta$-spaced real numbers, that is
$\|\gamma_1-\gamma_2\|\ge \delta$ for all $\gamma_1,\gamma_2\in \Gamma$ with $\gamma_1\not=\gamma_2$. Let
$$S(\gamma)=\sum_{1\le m\le N}a(m)e(m\gamma),$$
where $a(m)$ are complex numbers. Then one has
$$\sum_{\gamma\in \Gamma}|S(\gamma)|^2\ll (N+\delta^{-1})\sum_{1\le m\le N}|a(m)|^2.$$
 \end{lemma}

 \begin{lemma}\label{largesieve2}Let
$$h(\gamma)=\sum_{1\le m\le N^{1/k}}b(m)e(m^k\gamma),$$
where $b(m)$ are complex numbers. Let $s$ be a positive integer. Then uniformly for $\beta\in \R$, one has
$$\sum_{q\le Q}\sum_{\substack{a=1 \\ (a,q)=1}}^q|h(\frac{a}{q}+\beta)|^{2s}\ll (N+Q^2)\int_0^1|h(\alpha)|^{2s}d\alpha.$$
 \end{lemma}
 \begin{proof}Let
 \begin{align*}a(m)=\int_0^1h(\alpha)^se(-m\alpha)d\alpha.\end{align*}
 Then we can represent $h(\frac{a}{q}+\beta)^s$ in the form
 \begin{align*}h(\frac{a}{q}+\beta)^s=\sum_{1\le m\le sN}a_{\beta}(m)e(\frac{a}{q}m),\end{align*}
 where
 \begin{align*}a_{\beta}(m)=a(m)e(m\beta).\end{align*}
 On applying Lemma \ref{largesieve1} with $\Gamma=\{a/q:\ 1\le a\le q\le Q, (a,q)=1\}$, we conclude that
\begin{align}\label{applylarge1}\sum_{q\le Q}\sum_{\substack{a=1 \\ (a,q)=1}}^q|h(\frac{a}{q}+\beta)|^{2s}\ll (N+Q^2)\sum_{1\le m\le sN}|a(m)|^2.\end{align}
Note that
\begin{align}\label{applylarge2}\sum_{1\le m\le sN}|a(m)|^2=\int_0^1|h(\alpha)|^{2s}d\alpha.\end{align}
The proof is complete by combining \eqref{applylarge1} and \eqref{applylarge2}.
 \end{proof}

 If we apply Lemma \ref{largesieve2} directly to deal with $\mathcal{J}_{k,s}(Q)$, then we may merely obtain the trivial bound (up to a very small power of $n$)
\begin{align}\label{trivialb}\mathcal{J}_{k,s}(Q)\ll \int_{0}^1|g_k(\alpha)|^{2s}d\alpha.\end{align}
In order to improve upon \eqref{trivialb}, we first prepare some lemmas.

Let $k\ge 2$ be fixed. Then each positive integer $d$ can be uniquely represented in the form
$$d=d_1d_2^2\cdots d_{k-1}^{k-1}d_k^k,$$
where the product $d_1d_2\cdots d_{k-1}$ is square-free, and we define
\begin{align}\label{definevarrho}\varrho(d):=\varrho_k(d)=d_1\cdots d_k.\end{align}
Note that $d|m^k$ is equivalent to
$$\varrho(d)|m.$$
We have the following.
\begin{lemma}One has
\begin{align}\label{boundsumoverm}\sum_{\substack{m\le M\\ d|m^k}}1\le \frac{M}{\varrho(d)}.\end{align}
\end{lemma}

Let
$$\mathcal{D}=\sum_{\substack{ d\le D}}d\varrho(d)^{-j}.$$
\begin{lemma}\label{lemmacalD}Let $j\ge k+1$. Then one has
$$\mathcal{D}\ll D^{\varepsilon}.$$
\end{lemma}
\begin{proof}We deduce from the definition of $\varrho(d)$ in \eqref{definevarrho} that
$$\mathcal{D}\ll \sum_{d_1,\ldots,d_{k}\le D}d_1^{1-j}d_{2}^{2-j}\cdots d_{k}^{k-j}.$$
Since $j\ge k+1$, one has
$$\mathcal{D}\ll D^{\varepsilon}.$$
This completes the proof.\end{proof}

Now we are ready to introduce the key lemma in this section. We choose  $t\in \Z^{+}$ such that
\begin{align}\label{introducet}(nQ^{-2})^{1/k}\le R^t< (nQ^{-2})^{1/k}R.\end{align}
By the definition of $g_k(\alpha)$ in \eqref{definegk} and the definition of $\tau_{t}(x)$ in \eqref{definetau}, we can represent $g_k$ in the form
\begin{align}\label{repgash}g_k(\alpha)= \sum_{x\le R^t}\tau_{t}(x)h(x^k\alpha),\end{align}
where
\begin{align*}h(\alpha)=\sum_{y\le R^{r-t}}\tau_{r-t}(y)
e(y^k\alpha).\end{align*}
Then we introduce
\begin{align}\label{defineXi}\Xi(\beta)=
\sum_{\substack{q\le Q}}\sum_{\substack{a=1 \\ (a,q)=1}}^q\big|h(\frac{a}{q}+\beta)\big|^{2s}.\end{align}

Let
\begin{align}\label{defineQ3}Q_3=n^{2^{-20}}.\end{align}
\begin{lemma}\label{lemmaprelargesieve}Suppose that $Q_3\le Q\le \frac{1}{4}n$. Let $s$ be a real number and $2s\ge k+1$. Then one has
\begin{align*}\mathcal{J}_{k,s}(Q)\ll n^{-1+\nu+\varepsilon}(nQ^{-2})^{\frac{2s}{k}}R^{2s}\sup_{\beta\in \R}\Xi(\beta).\end{align*}
\end{lemma}
\begin{proof}We have
\begin{align}\label{boundJks1}\mathcal{J}_{k,s}(Q)\le n^{-1+\nu}\sup_{|\theta|\le n^{\nu-1}}\Sigma(\theta),\end{align}
where
\begin{align}\label{defineSigmatheta}\Sigma(\theta)=\sum_{q\le Q}\sum_{\substack{a=1 \\ (a,q)=1}}^q|g_k(\frac{a}{q}+\theta)|^{2s}.\end{align}

 We deduce from \eqref{repgash} that
$$|g_k(\alpha)|\ll \sum_{x\le R^t}|h(x^k\alpha)|\ll \sum_{d|q}\sum_{\substack{x\le R^t\\ (x^k,q)=d}}|h(x^k\alpha)|.$$
By H\"older's inequality,
$$|g_k(\alpha)|^{2s}\ll \Big(\sum_{d|q}1\Big)^{2s-1}\sum_{d|q}\Big(\sum_{\substack{x\le R^t\\ (x^k,q)=d}}|h(x^k\alpha)|\Big)^{2s}.$$
We deduce, by the elementary inequality for the divisor function, that
$$|g_k(\alpha)|^{2s}\ll q^{\varepsilon}\sum_{d|q}\Big(\sum_{\substack{x\le R^t\\ (x^k,q)=d}}|h(x^k\alpha)|\Big)^{2s}.$$
On applying H\"older's inequality again, we have
$$|g_k(\alpha)|^{2s}\ll q^{\varepsilon}\sum_{\substack{d|q}}\Big(\sum_{\substack{x\le R^t\\ (x^k,q)=d}}1\Big)^{2s-1}
\sum_{\substack{x\le R^t\\ (x^k,q)=d}}|h(x^k\alpha)|^{2s},$$
whence by \eqref{boundsumoverm},
\begin{align}\label{boundgks}|g_k(\alpha)|^{2s}\ll q^{\varepsilon}R^{(2s-1)t}\sum_{\substack{d|q}}\varrho(d)^{-2s+1}
\sum_{\substack{x\le R^t\\ (x^k,q)=d}}|h(x^k\alpha)|^{2s}.\end{align}

We conclude from \eqref{defineSigmatheta} and \eqref{boundgks} that
\begin{align*}\Sigma(\theta)\ll n^{\varepsilon}R^{(2s-1)t}\sum_{\substack{ d\le Q}}\varrho(d)^{-2s+1}
\sum_{\substack{q\le Q
\\ d|q}}\sum_{\substack{a=1 \\ (a,q)=1}}^q\sum_{\substack{x\le R^t\\ (x^k,q)=d}}|h(\frac{ax^k}{q}+x^k\theta)|^{2s},\end{align*}
and by exchanging the order of the summations
\begin{align}\label{boundSigmatheta1}\Sigma(\theta)\ll n^{\varepsilon}R^{(2s-1)t}\sum_{\substack{ d\le Q}}\varrho(d)^{-2s+1}
\sum_{\substack{q\le Q
\\ d|q}}\sum_{\substack{x\le R^t\\ (x^k,q)=d}}\sum_{\substack{a=1 \\ (a,q)=1}}^q|h(\frac{ax^k/d}{q/d}+x^k\theta)|^{2s}.\end{align}

We deduce that
\begin{align}\label{sumovera1}\sum_{\substack{a=1 \\ (a,q)=1}}^q|h(\frac{ax^k/d}{q/d}+x^k\theta)|^{2s} \ll d \sum_{\substack{a=1 \\ (a,q/d)=1}}^{q/d}|h(\frac{ax^k/d}{q/d}+x^k\theta)|^{2s},\end{align}
and by $(x^k,q)=d$, we have $(x^k/d,q/d)=1$ and
\begin{align}\label{sumovera2} \sum_{\substack{a=1 \\ (a,q/d)=1}}^{q/d}|h(\frac{ax^k/d}{q/d}+x^k\theta)|^{2s}=
\sum_{\substack{a=1 \\ (a,q/d)=1}}^{q/d}|h(\frac{a}{q/d}+x^k\theta)|^{2s}.\end{align}
We conclude from \eqref{boundSigmatheta1}, \eqref{sumovera1} and \eqref{sumovera2} that
\begin{align}\label{boundSigmatheta2}\Sigma(\theta)\ll n^{\varepsilon}R^{(2s-1)t}\sum_{\substack{ d\le Q}}d\varrho(d)^{-2s+1}
\sum_{\substack{q\le Q
\\ d|q}}\sum_{\substack{x\le R^t\\ (x^k,q)=d}}\sum_{\substack{a=1 \\ (a,q/d)=1}}^{q/d}|h(\frac{a}{q/d}+x^k\theta)|^{2s}.\end{align}
Now we replace the condition $(x^k,q)=d$ in \eqref{boundSigmatheta2} by $d|x^k$  to deduce that
\begin{align}\label{boundSigmatheta3}\Sigma(\theta)\ll n^{\varepsilon}R^{(2s-1)t}
\sum_{\substack{ d\le Q}}d\varrho(d)^{-2s+1}\sum_{\substack{x\le R^t\\ d|x^k }}
\sum_{\substack{q\le Q
\\ d|q}}\sum_{\substack{a=1 \\ (a,q/d)=1}}^{q/d}|h(\frac{a}{q/d}+x^k\theta)|^{2s}.\end{align}

By \eqref{defineXi}, we have
\begin{align*}\sum_{\substack{q\le Q
\\ d|q}}\sum_{\substack{a=1 \\ (a,q/d)=1}}^{q/d}|h(\frac{a}{q/d}+x^k\theta)|^{2s} \le \sup_{\beta}\Xi(\beta),\end{align*}
and by \eqref{boundSigmatheta3},
\begin{align*}\Sigma(\theta)\ll n^{\varepsilon}R^{(2s-1)t}\Big(\sup_{\beta}\Xi(\beta)\Big)
\sum_{\substack{ d\le Q}}d\varrho(d)^{-2s+1}\sum_{\substack{x\le R^t\\ d|x^k }}1.\end{align*}
On applying \eqref{boundsumoverm} again, we have
\begin{align*}\Sigma(\theta)\ll n^{\varepsilon}R^{2st}\Big(\sup_{\beta}\Xi(\beta)\Big)\sum_{\substack{ d\le Q}}d\varrho(d)^{-2s},\end{align*}
and by Lemma \ref{lemmacalD} and \eqref{introducet},
\begin{align}\label{boundSigma}\Sigma(\theta)\ll n^{\varepsilon}(nQ^{-2})^{\frac{2s}{k}}R^{2s}\sup_{\beta}\Xi(\beta).\end{align}
The proof is complete by combining \eqref{boundJks1} and \eqref{boundSigma}.\end{proof}

Now it is time to apply the large sieve inequality.
\begin{lemma}\label{lemmaJks}Suppose that $Q_3\le Q\le \frac{1}{4}n$. Let $s$ be a positive integer and $2s\ge k+1$. Suppose further that
$\lambda_{k,s}$ is a permissible exponent. Then we have
\begin{align}\label{newboundJks}\mathcal{J}_{k,s}(Q)\ll n^{-1+\frac{2s}{k}+\nu+\varepsilon}R^{2s}Q^{\frac{2(k-2s+\lambda_{k,s})}{k}}.\end{align}
\end{lemma}
\begin{proof}
In view of \eqref{introducet}, we have
$$R^{r-t}\le Q^{2/k}.$$
On applying Lemma \ref{largesieve2}, we conclude that
\begin{align*}\sup_{\beta\in \R}\Xi(\beta) \ll Q^2\int_0^1|h(\alpha)|^{2s}d\alpha.\end{align*}
Since $\lambda_{k,s}$ is a permissible exponent, on considering the solutions of the underlying diophantine equations, we have
$$\int_0^1|h(\alpha)|^{2s}d\alpha\ll R^{(r-t)\lambda_{k,s}}$$
 and by the above
\begin{align*}\sup_{\beta\in \R}\Xi(\beta) \ll Q^{2+2\frac{\lambda_{k,s}}{k}}.\end{align*}

Now we deduce from Lemma \ref{lemmaprelargesieve} that
\begin{align*}\mathcal{J}_{k,s}(Q)\ll n^{-1+\nu+\varepsilon}(nQ^{-2})^{\frac{2s}{k}}R^{2s}Q^{2+2\frac{\lambda_{k,s}}{k}}
\ll n^{-1+\frac{2s}{k}+\nu+\varepsilon}R^{2s}Q^{\frac{2(k-2s+\lambda_{k,s})}{k}}.\end{align*}
This completes the proof.
\end{proof}

For comparison, one may conventionally deduce from \eqref{trivialb} that
\begin{align}\label{trivialb2}\mathcal{J}_{k,s}(Q)\ll n^{\frac{\lambda_{k,s}}{k}}.\end{align}
The right hand side of the inequality \eqref{newboundJks} in the case $Q=n^{1/2}$ coincides with $n^{\frac{\lambda_{k,s}}{k}}$ up to a very small power of $n$.
Since one has $2s-\lambda_{k,s}<k$ in our applications, the estimate \eqref{newboundJks} improves upon \eqref{trivialb2} as soon as $Q\le n^{1/2-\delta}$ for some small $\delta>0$.

Lemma \ref{lemmaJks} holds with $Q_3$ replaced by $n^{\delta}$ for any $\delta>0$, while we may need to assume $R\le n^{\eta}$ for some $\eta$ sufficiently small in terms of $\delta$.

\vskip3mm

\section{The proof of Lemma \ref{lemmaintmQ2}}

Let
$$S_k(q,a)=\sum_{x=1}^qe(ax^k/q).$$
For fixed $k\in \Z^{+}$, we define the multiplicative function $\omega_k(q)$, by taking
$$ \omega_k(p^{ku+v})=\begin{cases}kp^{-u-1/2}, &\mbox{when $u\ge0$ and $v=1$,} \\
p^{-u-1}, & \mbox{when $u\ge0$ and $2\le v\le k$} \end{cases} $$
for prime powers. Note that
\begin{align}\label{boundomega}q^{-\frac{1}{2}}\le \omega_k(q)\ll q^{-\frac{1}{k}+\varepsilon}.\end{align}
\begin{lemma}\label{lemmagauss1}Suppose that $(a,q)=1$. Then we have
$$S_k(q,a)\ll q\omega_k(q).$$
In particular, one has
$$S_k(q,a)\ll q^{1-\frac{1}{k}+\varepsilon}.$$
\end{lemma}
\begin{proof}This follows from Lemma 3 of Vaughan \cite{Vaughan1986}. One can also refer to (2.3) in \cite{KW}.\end{proof}

In view of Lemma \ref{lemmaintmQ1}, we only need to consider the integration over $\mathfrak{M}(Q_1)$. From now on, throughout this paper,
we assume that $\alpha\in \mathfrak{M}(Q_1)$ has the unique representation
\begin{align}\label{representalpha}\alpha=\frac{a}{q}+\beta\ \ \textrm{ with }\  1\le a\le q\le Q_1,\ (a,q)=1 \ \textrm{ and } \ |\beta|\le \frac{Q_1}{qn}.\end{align}
For $k\in \{2,3\}$, we define the function $F_k^\ast(\alpha)$ on $\mathfrak{M}(Q_1)$ by
\begin{align*}F_k^\ast(\alpha)=\frac{1}{q}S_k(q,a)v_k(\beta),\end{align*}
where
$$v_k(\beta)=\int_{X_k/2}^{X_k} w(x/X_k)e(x^k\beta)dx.$$
\begin{lemma}\label{lemmaboundvk}For any $j\in \Z^{+}$, one has
$$v_k(\beta)\ll_j X_k(1+n|\beta|)^{-j}.$$\end{lemma}
\begin{proof}This follows easily from the integration by parts in the case $|\beta|\ge n^{-1}$ in combination with the trivial bound in the case
$|\beta|< n^{-1}$. \end{proof}

We define the function $\Delta_2(\alpha)$ on $\mathfrak{M}(Q_1)$ by
$$\Delta_2(\alpha)=F_2(\alpha)-F_2^\ast(\alpha).$$
One can conclude from Theorem 4.1 in \cite{V} and the partial summation formula that $\Delta_2(\alpha)\ll Q_1^{1/2+\varepsilon}$ for $\alpha\in \mathfrak{M}(Q_1)$. One may deduce much stronger estimate for $\Delta_2(\alpha)$ by applying Poisson's summation formula. We arrive at the following.
\begin{lemma}\label{lemmaintMQ1}One has
\begin{align}\label{usingDelta2}\int_{\mathfrak{M}(Q_1)}|\Delta_2(\alpha)F_3(\alpha)F(\alpha)G(\alpha)|d\alpha \ll \mathcal{F}(0)n^{-1-\frac{1}{2}\delta_1}.
\end{align}
\end{lemma}
\begin{proof}This follows from
$$\int_{\mathfrak{M}(Q_1)}|\Delta_2(\alpha)F_3(\alpha)F(\alpha)G(\alpha)|d\alpha \ll Q_1^{1/2+\varepsilon}
\int_0^1|F_3(\alpha)F(\alpha)G(\alpha)|d\alpha$$
in combination with \eqref{boundF3FG}.
Indeed, we obtain a better estimate than \eqref{usingDelta2}.
\end{proof}

Note that
$$\mathfrak{M}^\ast(Q)\subseteq \mathfrak{M}(Q).$$
The next lemma can be used to replace $\mathfrak{M}(Q)$ by $\mathfrak{M}^\ast(Q)$ in the integration.
\begin{lemma}\label{lemmaintMastQ1}Let $Q\le Q_1$. Let $\mathcal{M}(Q)=\mathfrak{M}(Q)\setminus \mathfrak{M}^\ast(Q)$. Let $H(\alpha)$ be a continious function of period one. For any constant $A>0$, one has
$$\int_{\mathcal{M}(Q)}|F_2^\ast(\alpha)H(\alpha)|d\alpha \ll n^{-A}\sup_{\alpha}|H(\alpha)|.$$
\end{lemma}
\begin{proof}For $\alpha\in \mathcal{M}(Q)$, one has $|\beta|> n^{-1+\nu}$. Then by Lemma \ref{lemmaboundvk}, one has
$$v_2(\beta)\ll_A n^{-A}\ \ \textrm{ and }\ F_2^\ast(\alpha) \ll_A n^{-A}$$
for any constant $A>0$. This completes the proof.
\end{proof}

For $X\le Q_1$, we introduce
\begin{align}\label{defineUpsi}\Upsilon(X)=\int_{\mathfrak{M}(Q_1)\setminus \mathfrak{M}(X)}|F_2^\ast(\alpha)G(\alpha)^2|d\alpha.\end{align}

\begin{lemma}\label{lemmaUpsilon}Let $Q_3\le X\le Q_1$. Then one has
$$\Upsilon(X)\ll F_2(0)G(0)^2 n^{-1+\nu+\varepsilon}R^{22}X^{-2\delta_1}.$$
\end{lemma}
\begin{proof}On writing
$$\Upsilon_0(Q)=\int_{\mathfrak{M}(Q)\setminus \mathfrak{M}(Q/2)}|F_2^\ast(\alpha)G(\alpha)^2|d\alpha $$
for $X\le Q\le Q_1$, by the dyadic argument, we only need to prove
\begin{align}\label{boundUpsi0}\Upsilon_0(Q) \ll
F_2(0)G(0)^2 n^{-1+\nu+\varepsilon}R^{22}Q^{-2\delta_1}.\end{align}

In view of Lemma \ref{lemmaintMastQ1}, one has
\begin{align}\label{boundfromUpsitoast}\Upsilon_0(Q)=\Upsilon_0^\ast(Q)+O(n^{-A}),\end{align}
where
$$\Upsilon_0^\ast(Q)=\int_{\mathfrak{M}^\ast(Q)\setminus \mathfrak{M}(Q/2)}|F_2^\ast(\alpha)G(\alpha)^2|d\alpha.$$
By Lemma \ref{lemmagauss1} and Lemma \ref{lemmaboundvk}, for $\alpha\in \mathfrak{M}^\ast(Q)\setminus \mathfrak{M}(Q/2)$, we have
$$F_2^\ast(\alpha)\ll F_2(0)Q^{-\frac{1}{2}+\varepsilon},$$
and therefore,
\begin{align}\label{boundUpsiast}\Upsilon_0^\ast(Q)\ll F_2(0)Q^{-\frac{1}{2}+\varepsilon}\int_{\mathfrak{M}^\ast(Q)}|G(\alpha)^2|d\alpha.\end{align}

We deduce by H\"older's inequality that
$$\int_{\mathfrak{M}^\ast(Q)}|G(\alpha)|^2d\alpha\le \mathcal{J}_{7,6}(Q)^{\frac{7}{s_7}-1}\mathcal{J}_{7,7}(Q)^{1-\frac{6}{s_7}}\prod_{k\in K_1\setminus\{7\}}\mathcal{J}_{k,s_k}(Q)^{\frac{1}{s_k}},$$
where $s_k$ are given in \eqref{definesk1} and \eqref{definesk2}. On applying Lemma \ref{lemmaJks}, we deduce that
$$\int_{\mathfrak{M}^\ast(Q)}|G(\alpha)|^2d\alpha\ll G(0)^2n^{-1+\nu+\varepsilon}R^{22}Q^{2-2\alpha(K_1)},$$
where $\alpha(K_1)$ is given in \eqref{definealphaK1}. Since $\alpha(K_1)=\frac{3}{4}+\delta_1$, we have
\begin{align}\label{boundUpsiinner}\int_{\mathfrak{M}^\ast(Q)}|G(\alpha)|^2d\alpha\ll G(0)^2n^{-1+\nu+\varepsilon}R^{22}Q^{\frac{1}{2}-2\delta_1}.\end{align}
Now \eqref{boundUpsi0} follows from \eqref{boundfromUpsitoast}, \eqref{boundUpsiast} and \eqref{boundUpsiinner}. The proof of the lemma is complete.
\end{proof}

\begin{lemma}\label{lemmaintMQ1toQ2}Let $Q_2$ be given in \eqref{defineQ2}.
 Then we have
$$\int_{\mathfrak{M}(Q_1)\cap \mathfrak{m}(Q_2)}|F_2^\ast(\alpha)F_3(\alpha)F(\alpha)G(\alpha)|d\alpha \ll \mathcal{F}(0)n^{-1+\nu+\varepsilon}R^{11}Q_2^{-\delta_1}.$$
\end{lemma}
\begin{proof}
By Lemma \ref{lemmaboundF2}, we have
$$\int_0^1|F_3(\alpha)F(\alpha)|^2d\alpha\ll F_3(0)^2F(0)^2n^{-\frac{7}{9}+\rho}.$$
By Lemma \ref{lemmagauss1} and Lemma \ref{lemmaboundvk}, we have
$$\sup_{\alpha \in \mathfrak{m}(Q_2)}|F_2^\ast(\alpha)|\ll F_2(0)Q_2^{-\frac{1}{2}+\varepsilon}.$$
Then we deduce that
$$\int_{\mathfrak{m}(Q_2)}|F_2^\ast(\alpha)F_3(\alpha)^2F(\alpha)^2|d\alpha\ll
F_2(0)F_3(0)^2F(0)^2Q_2^{-\frac{1}{2}+\varepsilon}n^{-\frac{7}{9}+\rho},$$
and by \eqref{defineQ2},
\begin{align}\label{boundF23F}\int_{\mathfrak{m}(Q_2)}|F_2^\ast(\alpha)F_3(\alpha)^2F(\alpha)^2|d\alpha\ll
F_2(0)F_3(0)^2F(0)^2n^{-1+\varepsilon}.\end{align}

We write $\mathfrak{R}=\mathfrak{M}(Q_1)\cap \mathfrak{m}(Q_2)$. On applying Schwarz's inequality, we deduce that
\begin{align*}\int_{\mathfrak{R}}|F_2^\ast(\alpha)F_3(\alpha)F(\alpha)G(\alpha)|d\alpha
 \ll
\Big(\int_{\mathfrak{m}(Q_2)}|F_2^\ast(\alpha)F_3(\alpha)^2F(\alpha)^2|d\alpha\Big)^{1/2} \Upsilon(Q_2)^{1/2},\end{align*}
and by \eqref{boundF23F},
\begin{align}\label{usingSch} \int_{\mathfrak{R}}|F_2^\ast(\alpha)F_3(\alpha)F(\alpha)G(\alpha)|d\alpha
\ll
\Big(F_2(0)F_3(0)^2F(0)^2n^{-1+\varepsilon}\Big)^{1/2} \Upsilon(Q_2)^{1/2}.\end{align}
On applying Lemma \ref{lemmaUpsilon}, we have
\begin{align}\label{usingUpsilonQ2}\Upsilon(Q_2)\ll F_2(0)G(0)^2 n^{-1+\nu+\varepsilon}R^{22}Q_2^{-2\delta_1}.
\end{align}
The proof is complete by combining \eqref{usingSch} and \eqref{usingUpsilonQ2}.
\end{proof}

\noindent {\it Proof of Lemma \ref{lemmaintmQ2}.} We conclude from Lemma \ref{lemmaintMQ1toQ2} that
\begin{align}\label{boundfromQ1toQ2}\int_{\mathfrak{M}(Q_1)\cap \mathfrak{m}(Q_2)}|F_2^\ast(\alpha)F_3(\alpha)F(\alpha)G(\alpha)|d\alpha \ll \mathcal{F}(0)n^{-1-\frac{4}{9}\delta_1}.\end{align}
Note that $F_2(\alpha)=F_2^\ast(\alpha)+\Delta_2(\alpha)$. We deduce from \eqref{usingDelta2} and \eqref{boundfromQ1toQ2} that
\begin{align*}\int_{\mathfrak{M}(Q_1)\cap \mathfrak{m}(Q_2)}|F_2(\alpha)F_3(\alpha)F(\alpha)G(\alpha)|d\alpha \ll \mathcal{F}(0)n^{-1-\frac{4}{9}\delta_1}.\end{align*}
This completes the proof of Lemma \ref{lemmaintmQ2}.
\vskip3mm

\section{Breaking the classical convexity}

Let $K_2=\{4,12,13,14\}$ and
\begin{align*}\mathcal{K}_2=\int_0^{1}\prod_{k\in K_2}|f_{k}(\alpha;N^{1/k},R)|^4d\alpha.\end{align*}
We define
$$ \kappa_2=\sum_{k\in K_2}\frac{4}{k},$$
and one has the trivial inequality
$$\mathcal{K}_2\ll N^{\kappa_2}.$$

We remark that we only make use of mean value estimates of even moments to deal with $\mathcal{K}_1$. In order to handle
$\mathcal{K}_2$, we need the seventh moment of a smooth Weyl sum.

\begin{lemma}\label{lemmaBW}One has
$$\int_0^1|f_4(\alpha;N^{1/4},R)|^7d\alpha\ll N^{\frac{3.849408}{4}}.$$ \end{lemma}
\begin{proof}This follows from Theorem 2 of Br\"udern and Wooley \cite{BW}.\end{proof}
Let
$$\alpha_{4,7/2}=\frac{7-3.849408}{4}=0.787648.$$
As a consequence Lemma \ref{lemmaBW}, for $s\ge 7$, one has
\begin{align}\label{boundbreakcon}\int_0^1|f_4(\alpha;N^{1/4},R)|^sd\alpha\ll N^{\frac{s}{4}-\alpha_{4,7/2}}.\end{align}

\begin{lemma}\label{lemmabreakcon}One has
\begin{align*}\mathcal{K}_2 \ll N^{\kappa_2-0.7834034}.\end{align*}
\end{lemma}
\begin{proof}
We deduce by H\"older's inequality that
\begin{align}\label{boundcalK2}\mathcal{K}_2\le \prod_{k\in K_2}\Big(\int_{0}^1|f_{k}(\alpha;N^{1/k},R)|^{4s_k}d\alpha\Big)^{\frac{1}{s_k}},\end{align}
where $s_{12}=13/2,s_{13}=7,s_{14}=15/2$ and $s_{4}$ is determined by
$$\sum_{k\in K_2}\frac{1}{s_k}=1.$$
Note that $4s_4=7.0179948$, and in particular $4s_4>7$.

One has permissible exponents (see Tables in Sections 14-16 in \cite{VW2})
$$\lambda_{12,13}=16.6110110,\ \lambda_{13,14}=17.8953488,\ \lambda_{14,15}=19.1785686,$$
whence, by \eqref{definealphaks},
$$\alpha_{12,13}=0.7824157,\ \alpha_{13,14}=0.7772808,\ \alpha_{14,15}=0.7729593.$$

Therefore, by \eqref{boundfks-alpha}, \eqref{boundbreakcon} and \eqref{boundcalK2}, we have
\begin{align*}\mathcal{K}_2\ll
N^{\kappa_2-\alpha(K_2)},\end{align*}
where
\begin{align*}\alpha(K_2)=\sum_{k\in K_2\setminus\{4\}}\frac{\alpha_{k,2s_k}}{s_k}+\frac{\alpha_{4,7/2}}{s_4}.\end{align*}
Note that
$$\alpha(K_2)=0.7834034.$$
This completes the proof.
\end{proof}

\vskip3mm

\section{The fourth moment of a cubic Weyl sum}

In this section, we apply an estimate on the fourth moment of a cubic exponential sum to deal with
$$\int_{\mathfrak{M}(Q_2)\setminus \mathfrak{M}(n^{2/9})}|F_2^\ast(\alpha)F_2(\alpha)F(\alpha)G(\alpha)|d\alpha.$$
We need the following interesting result proved by Br\"udern \cite{Br}.
\begin{lemma}\label{lemmaB}One has
\begin{align*}\int_{\mathfrak{M}(Q)}|F_3(\alpha)|^4d\alpha \ll n^{\varepsilon}(n^{1/3}+Q^{7/2}n^{-1}+Q^2n^{-\frac{1}{3}}).\end{align*}\end{lemma}
\begin{proof}This is essentially Theorem 2 of Br\"udern \cite{Br}. Note that the definition of $w$ in \eqref{definew} is slightly different from the function $\Gamma$ used in \cite{Br}, while the proof is the same.\end{proof}

Now we prove the following.
\begin{lemma}\label{lemmaintMQ2to29}One has
$$\int_{\mathfrak{M}(Q_2)\cap \mathfrak{m}(n^{2/9})}|F_2^\ast(\alpha)F_3(\alpha)F(\alpha)G(\alpha)|d\alpha \ll \mathcal{F}(0)n^{-1-\frac{2}{9}\delta_1}.$$
\end{lemma}
\begin{proof}
Suppose that $n^{2/9}\le Q\le Q_2$. By Lemma \ref{lemmaintMastQ1}, we have
\begin{align}\label{boundfromMtoMast}\int_{\mathfrak{M}(Q)\setminus \mathfrak{M}(Q/2)}|F_2^\ast(\alpha)F_3(\alpha)F(\alpha)G(\alpha)|d
\alpha=\Theta(Q)+O(n^{-A}),\end{align}
where
\begin{align*}\Theta(Q)=\int_{\mathfrak{M}^\ast(Q)\setminus \mathfrak{M}(Q/2)}|F_2^\ast(\alpha)F_3(\alpha)F(\alpha)G(\alpha)|d
\alpha.\end{align*}
By H\"older's inequality,
\begin{align}\label{thetato12ups}\Theta(Q)\le \Theta_1(Q)^{1/4} \Theta_2(Q)^{1/4} \Upsilon(Q/2)^{1/2},\end{align}
where $\Upsilon(Q)$ is defined in \eqref{defineUpsi},
\begin{align*}\Theta_1(Q)=
\int_{\mathfrak{M}^\ast(Q)\setminus \mathfrak{M}(Q/2)}|F_2^\ast(\alpha)^2F_3(\alpha)^4|d\alpha\end{align*}
and
\begin{align*}\Theta_2(Q)=
\int_{\mathfrak{M}^\ast(Q)}|F(\alpha)|^4d\alpha.\end{align*}

One has
$$F_2^\ast(\alpha)\ll F_2(0)Q^{-\frac{1}{2}+\varepsilon}$$
for $\alpha\in \mathfrak{M}^\ast(Q)\setminus \mathfrak{M}(Q/2)$. Then we deduce that
\begin{align*}\Theta_1(Q)\ll F_2(0)^2Q^{-1+\varepsilon}
\int_{\mathfrak{M}(Q)}|F_3(\alpha)^4|d\alpha.\end{align*}
By Lemma \ref{lemmaB},
\begin{align*}
\int_{\mathfrak{M}(Q)}|F_3(\alpha)^4|d\alpha \ll n^{\varepsilon}(n^{1/3}+Q^{7/2}n^{-1}+Q^2n^{-\frac{1}{3}}).\end{align*}
We conclude that
\begin{align*}\Theta_1(Q)\ll n^{\varepsilon}F_2(0)^2(Q^{-1}n^{1/3}+Q^{5/2}n^{-1}+Qn^{-\frac{1}{3}}).\end{align*}
Since $n^{2/9}\le Q\le Q_2$, we deduce that
\begin{align*}\Theta_1(Q)\ll n^{\varepsilon}F_2(0)^2(n^{1/9}+Q_2^{5/2}n^{-1}+Q_2n^{-\frac{1}{3}})\ll n^{\varepsilon}F_2(0)^2Q_2^{5/2}n^{-1},\end{align*}
and in particular,
\begin{align}\label{boundTheta1}\Theta_1(Q)\ll F_2(0)^2F_3(0)^4n^{-1-\frac{2}{9}+5\rho+\varepsilon}.\end{align}

Note that $Y_k^k= n^{\lambda}$. We can represent $F^2=f_{4}^2f_{12}^2f_{13}^2f_{14}^2$ in the form
\begin{align*}F(\frac{a}{q}+\beta)^2=\sum_{1\le m\le 8n^{\lambda}}a_{\beta}(m)e(m\frac{a}{q}),\end{align*}
where $a_{\beta}(m)=a(m)e(m\beta)$ and
\begin{align*}a(m)=\int_{0}^1F(\alpha)^2e(-m\alpha)d\alpha.\end{align*}
We deduce that
\begin{align*}\Theta_2(Q)\ll  n^{\nu-1}\sup_{\beta} \sum_{q\le Q}\sum_{\substack{a=1 \\ (a,q)=1}}^q \Big|\sum_{1\le m\le 8n^{\lambda}}a_{\beta}(m)e(m\frac{a}{q})\Big|^2,\end{align*}
and on applying Lemma \ref{largesieve1} with $\Gamma=\{a/q:\ 1\le a\le q\le Q, (a,q)=1\}$, we further deduce that
\begin{align*}\Theta_2(Q)\ll n^{\nu-1}(n^{\lambda}+Q^2) \sum_{1\le m\le 8n^{\lambda}}|a(m)|^2
\ll  n^{\nu-1}(Q^2+n^{\lambda})
\int_{0}^1|F(\alpha)|^4d\alpha.\end{align*}
Since $Q^2\le Q_2^2\le n^{\lambda}$, we conclude that
\begin{align*}\Theta_2(Q)\ll n^{\nu-1+\lambda}
\int_{0}^1|F(\alpha)|^4d\alpha.\end{align*}
On applying Lemma \ref{lemmabreakcon} with $N=n^{\lambda}$, we deduce that
\begin{align*}
\int_{0}^1|F(\alpha)|^4d\alpha \ll F(0)^4n^{-\lambda \alpha_2},\end{align*}
where
$$\alpha_2=0.7834034.$$
Therefore,
\begin{align}\label{boundTheta2}\Theta_2(Q)\ll F(0)^4n^{\nu-1+\lambda-\lambda \alpha_2}.\end{align}

On applying Lemma \ref{lemmaUpsilon}, we have
\begin{align}\label{boundUinTheta}\Upsilon(Q/2)\ll G(0)^2n^{-1+\nu+\varepsilon}R^{22}Q^{-2\delta_1}.\end{align}
It follows from \eqref{thetato12ups}, \eqref{boundTheta1}, \eqref{boundTheta2} and \eqref{boundUinTheta} that
\begin{align}\label{boundThetawithdelta2}\Theta(Q)\ll \mathcal{F}(0)n^{\nu+\varepsilon-1-\frac{1}{4}\delta_2}R^{11}Q^{-\delta_1},\end{align}
where
\begin{align*}\delta_2=\frac{2}{9}-5\rho+\lambda \alpha_2-\lambda.\end{align*}
Note that $\delta_2=0.001651382$, and in particular,
\begin{align}\label{boundTheta}\Theta(Q)\ll \mathcal{F}(0)n^{-1}Q^{-\delta_1}.\end{align}

We deduce from \eqref{boundfromMtoMast} and \eqref{boundTheta} that
\begin{align*}\int_{\mathfrak{M}(Q)\setminus \mathfrak{M}(Q/2)}|F_2^\ast(\alpha)F_3(\alpha)F(\alpha)G(\alpha)|d
\alpha \ll \mathcal{F}(0)n^{-1}Q^{-\delta_1}.\end{align*}
The proof is now complete by the standard dyadic argument.
\end{proof}

We remark that Lemma \ref{lemmaintm29} follows immediately from Lemma \ref{usingDelta2} and Lemma \ref{lemmaintMQ2to29} since $F_2(\alpha)=F_2^\ast(\alpha)+\Delta_2(\alpha)$.
Let
$$\mathcal{N}_1(n)=\int_{\mathfrak{M}(n^{2/9})}F_2^\ast(\alpha)F_3(\alpha)F(\alpha)G(\alpha)e(-n\alpha)d\alpha.$$
\begin{lemma}\label{lemmaN1}One has
$$\mathcal{N}(n)-\mathcal{N}_1(n) \ll \mathcal{F}(0)n^{-1-\frac{2}{9}\delta_1}.$$
\end{lemma}
\begin{proof}This follows from Lemma \ref{usingDelta2} and Lemmas \ref{lemmaintmQ1}-\ref{lemmaintm29}.\end{proof}
It remains to establish the asymptotic formula of $\mathcal{N}_1(n)$. In fact, it is not difficult to obtain the asymptotic formula of
$\int_{\mathfrak{M}(n^{2/9})}F_2^\ast(\alpha)F_3(\alpha)G(\alpha)e(-n\alpha)d\alpha$.
The arguments in the next two sections are routine.

\vskip3mm

\section{The pruning argument}

The following lemma is due to Br\"udern (see also Lemma 4.5 in Ford \cite{Ford}).
\begin{lemma}[Br\"udern] \label{lemmabrudern}Let $Q\le N$. For $1\le a\le q\le Q$, $(a,q)=1$, let $\mathcal{M}(q,a)$ denote an interval contained in $[a/q-1/2,a/q+1/2]$ and assume that $\mathcal{M}(q,a)$ are pairwise disjoint. Write $\mathcal{M}$ for the union of all $\mathcal{M}(q,a)$. Let $\mathcal{H}:\mathcal{M}\rightarrow\C$ be a function satisfying
$$\mathcal{H}(\alpha)\ll \frac{N}{q(1+N|\alpha-a/q|)}$$
for $\alpha\in \mathcal{M}(q,a)$. Let $\widetilde{\Psi}:\R\rightarrow [0,+\infty)$ be a function with a Fourier expansion
$$\widetilde{\Psi}(\alpha)=\sum_{|h|\le H}\psi_he(\alpha h)$$
and $\log H\ll \log N$. Then
$$\int_{\mathcal{M}}\mathcal{H}(\alpha)\widetilde{\Psi}(\alpha)\ll Q\psi_0 \log N+\sum_{0<|h|\le H}|\psi_h|d(|h|),$$
where $d(\cdot)$ denotes the divisor function. \end{lemma}

We define the function $\mathcal{G}:\mathfrak{M}(n^{2/9})\rightarrow\R^{+}$ by
$$\mathcal{G}(\alpha)= \frac{n}{q(1+n|\alpha-a/q|)}$$
for $\alpha\in \mathfrak{M}(q,a;n^{2/9})$ with $1\le a\le q\le n^{2/9}$ and $(a,q)=1$.

\begin{lemma}One has
\begin{align}\label{ineqB1}\int_{\mathfrak{M}(n^{2/9})}\mathcal{G}(\alpha)|F(\alpha)|^2d\alpha \ll
n^{\varepsilon}F(0)^2\end{align}
and
\begin{align}\label{ineqB2}\int_{\mathfrak{M}(n^{2/9})}|F(\alpha)|^2d\alpha \ll
F(0)^2n^{-7/9+\varepsilon}.\end{align}
\end{lemma}
\begin{proof}We represent $|F(\alpha)|^2$ in the form
$$|F(\alpha)|^2=\sum_{|h|\le 4n^{\lambda}}\psi_h e(h\alpha),$$
where
$$\psi_h=\int_0^1|F(\alpha)|^2e(-h\alpha)d\alpha.$$

On applying Lemma \ref{lemmabrudern}, we deduce that
\begin{align*}\int_{\mathfrak{M}(n^{1/3})}\mathcal{G}(\alpha)|F(\alpha)|^2d\alpha
\ll &n^{2/9}\psi_0 \log n+\sum_{0<|h|\le H}|\psi_h|d(|h|)
\\ \ll &n^{2/9+\varepsilon}\psi_0 + n^{\varepsilon}F(0)^2.\end{align*}
On recalling \eqref{defineIj}, we have $\psi_0=\mathcal{I}_1$. By \eqref{boundI1}, one has
\begin{align*}n^{2/9}\psi_0\ll n^{2/9+u_1+\varepsilon} \ll  F(0)^2.\end{align*}
This completes the proof of \eqref{ineqB1}. Then \eqref{ineqB2} follows from \eqref{ineqB1} by noting that
$\mathcal{G}(\alpha)\ge n^{7/9}$ for $\alpha\in \mathfrak{M}(n^{2/9})$.
\end{proof}

We also use the following result due to McDonagh \cite{McD}.
\begin{lemma}\label{lemmaMcD}For any $k\in \Z^{+}$, there exists a constant $C_k>0$ such that
$$\sum_{1\le x<N^{1/k}}d(N-x^k)\ll N^{1/k}(\log N)^{C_k}.$$
\end{lemma}

\begin{lemma}\label{lemmaapplyB2} Let $Q_3$ be given in \eqref{defineQ3}. Let $k\in \{12,13\}$. Then there exists a constant $C>0$ such that
\begin{align}\label{ineqB3}\int_{\mathfrak{M}(Q_3)}\mathcal{G}(\alpha)|f_{k}(\alpha)|^2d\alpha \ll
f_k(0)^2(\log n)^{C}.\end{align}
\end{lemma}
\begin{proof}The proof is as same as \eqref{ineqB1} except that we use Lemma \ref{lemmaMcD} to show that the divisor function behaves like a power of $\log n$ in average. We omit the details.\end{proof}

\begin{lemma}\label{lemmaMn29}Let $Q_3$ be given in \eqref{defineQ3}. Then one has
$$\int_{\mathfrak{M}(n^{2/9})\setminus\mathfrak{M}(Q_3)}|F_2^\ast(\alpha)F_3(\alpha)F(\alpha)G(\alpha)|d\alpha \ll \mathcal{F}(0)n^{-1+\varepsilon}Q_3^{-\frac{1}{12}}.$$
\end{lemma}
\begin{proof}
By Schwarz's inequality,
\begin{align}\label{SchwithTheta3}\int_{\mathfrak{M}(n^{2/9})\setminus\mathfrak{M}(Q_3)}|F_2^\ast(\alpha)F_3(\alpha)F(\alpha)G(\alpha)|d\alpha \le  \Theta_3^{1/2}\Upsilon(Q_3)^{1/2},\end{align}
where
$$\Theta_3=\int_{\mathfrak{M}(n^{2/9})\setminus\mathfrak{M}(Q_3)}|F_2^\ast(\alpha)F_3(\alpha)^2F(\alpha)^2|d\alpha$$
and  $\Upsilon(Q_3)$ is defined in \eqref{defineUpsi}.

We define the function $\Delta_3(\alpha)$ on $\mathfrak{M}(n^{2/9})$ by
$$\Delta_3(\alpha)=F_3(\alpha)-F_3^\ast(\alpha).$$
It follows from Theorem 4.1 in \cite{V} and the partial summation formula that
\begin{align}\label{boundD3}\Delta_3(\alpha)\ll Q^{\frac{1}{2}+\varepsilon}\ \textrm{ for }\ \alpha\in \mathfrak{M}(Q).\end{align}
Then we deduce from \eqref{ineqB2} and \eqref{boundD3} that
\begin{align*}\int_{\mathfrak{M}(n^{2/9})}|\Delta_3(\alpha)F(\alpha)|^2d\alpha \ll
F_3(0)^2F(0)^2n^{-1-\frac{2}{9}+\varepsilon},\end{align*}
and by the trivial bound $F_2^\ast(\alpha) \ll F_2(0)$,
\begin{align}\label{boundTheta3withD3}\int_{\mathfrak{M}(n^{2/9})}|F_2^\ast(\alpha)\Delta_3(\alpha)^2F(\alpha)^2|d\alpha \ll
F_2(0)F_3(0)^2F(0)^2n^{-1-\frac{2}{9}+\varepsilon}.\end{align}

 For $\alpha\in \mathfrak{M}(n^{2/9})\setminus\mathfrak{M}(Q_3)$, we have
 $$F_2^\ast(\alpha)F_3^\ast(\alpha)^2\ll F_2(0)F_3(0)^2Q_3^{-\frac{1}{6}+\varepsilon}n^{-1}\mathcal{G}(\alpha).$$
It follows from \eqref{ineqB1} that
\begin{align}\label{boundTheta3withF3ast}\int_{\mathfrak{M}(n^{2/9})\setminus\mathfrak{M}(Q_3)}|F_2^\ast(\alpha)
F_3^\ast(\alpha)^2F(\alpha)^2|d\alpha \ll
 F_2(0)F_3(0)^2F(0)^2n^{-1+\varepsilon}Q_3^{-\frac{1}{6}}.\end{align}
Then we conclude from \eqref{boundTheta3withD3} and \eqref{boundTheta3withF3ast} that
 \begin{align}\label{boundTheta3}\Theta_3\ll
 F_2(0)F_3(0)^2F(0)^2n^{-1+\varepsilon}Q_3^{-\frac{1}{6}}.\end{align}

 On applying Lemma \ref{lemmaUpsilon}, we deduce from \eqref{SchwithTheta3} and \eqref{boundTheta3} that
$$\int_{\mathfrak{M}(n^{2/9})\setminus\mathfrak{M}(Q_3)}|F_2^\ast(\alpha)F_3(\alpha)F(\alpha)G(\alpha)|d\alpha \ll \mathcal{F}(0)n^{-1+\varepsilon}Q_3^{-\frac{1}{12}}.$$
This completes the proof.
\end{proof}

Let
$$\mathcal{N}_2(n)=\int_{\mathfrak{M}(Q_3)}F_2^\ast(\alpha)F_3^\ast(\alpha)F(\alpha)G(\alpha)e(-n\alpha)d\alpha.$$
\begin{lemma}\label{lemmaN2}One has
$$\mathcal{N}(n)-\mathcal{N}_2(n) \ll \mathcal{F}(0)n^{-1+\varepsilon}Q_3^{-\frac{1}{12}}.$$
\end{lemma}
\begin{proof}In view of Lemma \ref{lemmaN1} and Lemma \ref{lemmaMn29}, we only need to prove
\begin{align}\label{boundMQ3D3}\int_{\mathfrak{M}(Q_3)}|F_2^\ast(\alpha)\Delta_3(\alpha)F(\alpha)G(\alpha)|d\alpha \ll \mathcal{F}(0)n^{-1+\varepsilon}Q_3^{-\frac{1}{12}}.\end{align}
By \eqref{boundD3}, one has $\Delta_3(\alpha)\ll Q_3^{1/2+\varepsilon}$, and \eqref{boundMQ3D3} now follows from the trivial bound
 $F_2^\ast(\alpha)F(\alpha)G(\alpha)\ll F_2(0)F(0)G(0)$.\end{proof}

In order to deal with $\mathcal{N}_2(n)$, we need upper bounds of $g_k(\alpha)$, which will be deduced from the following lemma.
\begin{lemma}\label{lemmaKW}Let $k$ be a positive integer with $k\ge 4$. Let $P,M,M',U,U'$ be real numbers satisfying
$$P^{1/2}\le M\le M'\le \frac{4}{3}M, \ P/M\le U\le U'\le \frac{4}{3}P/M.$$
Suppose that $(a_x)$ and $(b_y)$ are complex numbers satisfying $|a_x|\le 1$ and $|b_y|\le 1$.
Suppose further that $\alpha$ is a real number, and that there exists $a\in \Z$ and $q\in \Z^{+}$ with
$$(a,q)=1, 1\le q\le P^{k/2} \textrm{ and } |q\alpha-a|\le P^{-k/2}.$$
Then one has
\begin{align*}&\sum_{M<x\le M'}a_x\sum_{U<y\le U'}b_ye\big((x^ky^k)\alpha\big)
\\ & \ \ \ \ \ \ \ll PM^{\varepsilon-2^{-k}}+(PM)^{1/2}+\frac{q^\varepsilon \omega_k(q)^{1/2}P(\log P)^{4}}{(1+P^k|\alpha-a/q|)^{1/2}}.\end{align*}
\end{lemma}
\begin{proof}This follows from Lemma 3.1 of Kawada and Wooley \cite{KW}.\end{proof}

\begin{lemma}\label{lemmagk}Let $\alpha\in \mathfrak{M}(q,a;Q_3)$ with $1\le a\le q\le Q_3$ and $(a,q)=1$. Let $5\le k\le 11$. Then one has
\begin{align*}g_k(\alpha)\ll g_k(0)\frac{q^\varepsilon \omega_k(q)^{1/2}(\log n)^{4+(k\eta)^{-1}}}{(1+n|\alpha-a/q|)^{1/2}}.\end{align*}
\end{lemma}
\begin{proof}We choose $t\in \Z^{+}$ such that
\begin{align*}X_k^{1/2}R\le R^t< X_k^{1/2}R^2.\end{align*}
We can represent $g_k$ in the form
\begin{align*}g_k(\alpha)= \sum_{(R/2)^t<x\le R^t}\tau_{t}(x)\sum_{(R/2)^{r-t}<y\le R^{r-t}}\tau_{r-t}(y)e\big((xy)^k\alpha\big),\end{align*}
where $\tau_t(x)$ is defined in \eqref{definetau} and $r=r_k=(k\eta)^{-1}$.
On applying Lemma \ref{lemmaKW}, we conclude that
\begin{align*}g_k(\alpha) \ll & X_k^{1+\varepsilon-2^{-k-1}}+X_k^{3/4}R+\frac{q^\varepsilon \omega_k(q)^{1/2}X_k(\log n)^{4}}{(1+n|\alpha-a/q|)^{1/2}}
\\ \ll & X_k^{1-2^{-k-2}}+\frac{q^\varepsilon \omega_k(q)^{1/2}X_k(\log n)^{4}}{(1+n|\alpha-a/q|)^{1/2}}.\end{align*}
We remark that the length of the interval $[(R/2)^t,R^t]$ is larger than $R^t/4$, and one may need to use the dyadic argument before applying Lemma \ref{lemmaKW}. Of course, the proof of Lemma 3.1 in \cite{KW} works well to deal with intervals longer than $[M,\frac{4}{3}M]$.

Note that
\begin{align*}\frac{\omega_k(q)^{1/2}X_k}{(1+n|\alpha-a/q|)^{1/2}}\ge X_kQ_3^{-1/2},\end{align*}
and $Q_3\le X_k^{2^{-k-1}}$ for $k\le 11$.
We finally conclude that
\begin{align*}g_k(\alpha) \ll \frac{q^\varepsilon \omega_k(q)^{1/2}X_k(\log n)^{4}}{(1+n|\alpha-a/q|)^{1/2}}.\end{align*}
This completes the proof on recalling \eqref{boundg0}.
\end{proof}

Let \begin{align}\label{defineQ4}Q_4=(\log n)^{A},\end{align}
where $A$ is a sufficiently large constant (depending on $\eta$). For example, we may choose
$$A= 10^{100}\eta^{-100}(1+C)^{10},$$
where $C$ is the constant in \eqref{ineqB3}.
\begin{lemma}\label{lemmaintM3to4}One has
$$\int_{\mathfrak{M}(Q_3)\setminus\mathfrak{M}(Q_4)}
|F_2^\ast(\alpha)F_3^\ast(\alpha)F(\alpha)G(\alpha)|d\alpha \ll \mathcal{F}(0)n^{-1}Q_4^{-\frac{1}{5}}.$$
\end{lemma}
\begin{proof}One deduces  by \eqref{boundomega} that
$$\omega_2(q)\omega_3(q)\prod_{k=5}^{11}\omega_k(q)^{1/2}\ll q^{-1.301}.$$
Then for $\alpha\in \mathfrak{M}(Q_3)\setminus\mathfrak{M}(Q_4)$, by Lemma \ref{lemmagauss1}, Lemma \ref{lemmaboundvk} and Lemma \ref{lemmagk}, one has
\begin{align}\label{boundF2F3G}F_2^\ast(\alpha)F_3^\ast(\alpha)G(\alpha)\ll F_2(0)F_3(0)G(0)Q_4^{-\frac{3}{10}}n^{-1}(\log n)^{28+7\eta^{-1}}\mathcal{G}(\alpha).\end{align}

On applying Lemma \ref{lemmaapplyB2} and Schwarz's inequality, we deduce that
\begin{align}\label{ineqB4}\int_{\mathfrak{M}(Q_3)}\mathcal{G}(\alpha)|F(\alpha)|d\alpha \ll
F(0)(\log n)^{C}.\end{align}
It follows from \eqref{boundF2F3G} and \eqref{ineqB4} that
\begin{align*}\int_{\mathfrak{M}(Q_3)\setminus\mathfrak{M}(Q_4)}
|F_2^\ast(\alpha)F_3^\ast(\alpha)F(\alpha)G(\alpha)|d\alpha
\ll  \mathcal{F}(0)Q_4^{-\frac{3}{10}} n^{-1}
(\log n)^{28+7\eta^{-1}+C}.\end{align*}
This completes the proof since $A$ in \eqref{defineQ4} is sufficiently large.
\end{proof}

Let
$$S_k^\ast(q,a)=\sum_{\substack{x=1 \\ (x,q)=1}}^qe(ax^k/q).$$
\begin{lemma}\label{lemmagauss2}Suppose that $(a,q)=1$. Then we have
$$S_k^\ast(q,a)\ll q^{\frac{1}{2}+\varepsilon}.$$
\end{lemma}

For $5\le k\le 11$, on recalling the assumption \eqref{representalpha}, we define the function $\widetilde{g}_k(\alpha)$ on $\mathfrak{M}(Q_4)$ by
\begin{align*}\widetilde{g}_k(\alpha)=\frac{1}{\phi(q)}S^\ast_k(q,a)\widetilde{v}_k(\beta),\end{align*}
where $\phi(\cdot)$ is Euler's totient function and
$$\widetilde{v}_k(\beta)=\int_{[R/2,R]^{r_k}}\frac{e\big((x_1\cdots x_{r_k})^k\beta\big)}{\prod_{j=1}^{r_k}\log x_j}dx_1\cdots dx_{r_k}.$$
\begin{lemma}\label{lemmabounddeltag}Let $5\le k\le 11$. Let $\alpha\in \mathfrak{M}(Q_4)$. Then one has
$$g_k(\alpha)-\widetilde{g}_k(\alpha)\ll g_k(0)Q_4^{-50}.$$\end{lemma}
\begin{proof}Suppose that $|y^k|R^k\le n$. We can deduce from the Siegel-Walfisz theorem and the partial summation formula that
$$\sum_{R/2<p\le R}e\big((\frac{a}{q}+\beta)y^kp^k\big)=\frac{1}{\phi(q)}S^\ast(q,ay^k)\int_{R/2}^{R}\frac{e(t^ky^k\beta)}{\log t}dt+O(RQ_4^{-100}).$$
If $(y,q)=1$, then we have $S^\ast(q,ay^k)=S^\ast(q,a)$ and further deduce that
$$\sum_{R/2<p\le R}e\big((\frac{a}{q}+\beta)y^kp^k\big)=\frac{1}{\phi(q)}S^\ast(q,a)\int_{R/2}^{R}\frac{e(t^ky^k\beta)}{\log t}dt+O(RQ_4^{-100}).$$
Then on writing $s=r_k-1$, we conclude that
$$g_k(\alpha)=\frac{1}{\phi(q)}S^\ast(q,a)\sum_{R/2<p_1,\ldots,p_{s}\le R}\int_{R/2}^{R}\frac{e(t^k
p_1^k\cdots p_{s}^k\beta)}{\log t}dt+O(g_k(0)Q_4^{-99}).$$
On applying the partial summation formula $s$ times in combination with the prime number theorem, we deduce that
$$\sum_{R/2<p_1,\ldots,p_{s}\le R}\int_{R/2}^{R}\frac{e(t^kp_1^k\cdots p_{s}^k\beta)}{\log t}dt=\widetilde{v}_k(\beta)+O(g(0)Q_4^{-100}).$$
We conclude from above that
$$g_k(\alpha)=\frac{1}{\phi(q)}S^\ast(q,a)\widetilde{v}_k(\beta)+O(g_k(0)Q_4^{-99}).$$
This completes the proof.\end{proof}

For $5\le k\le 11$, we define the function $g_k^\ast(\alpha)$ on $\mathfrak{M}(Q_4)$ by
\begin{align*}g_k^\ast(\alpha)=\frac{1}{\phi(q)}S^\ast_k(q,a)v_k^\ast(\beta),\end{align*}
where
$$v_k^\ast(\beta)=\frac{1}{(\log R)^{r_k}}\int_{[R/2,R]^{r_k}}e\big((x_1\cdots x_{r_k})^k\beta\big)dx_1\cdots dx_{r_k}.$$
Define further
$$G^{\ast}(\alpha)=\prod_{k=5}^{11}g_k^\ast(\alpha)$$
and
$$\Delta_{G}(\alpha)=G(\alpha)-G^\ast(\alpha).$$
\begin{lemma}\label{lemmaDeltaG}One has
\begin{align}\label{boundDeltaG}\int_{\mathfrak{M}(Q_4)}
|F_2^\ast(\alpha)F_3^\ast(\alpha)F(\alpha)\Delta_{G}(\alpha)|d\alpha \ll \mathcal{F}(0)n^{-1}(\log n)^{-1}.\end{align}
\end{lemma}
\begin{proof}On writing
$$\widetilde{\Delta}_{G}(\alpha)=G(\alpha)-\prod_{k=5}^{11}\widetilde{g}_k(\alpha),$$
we deduce from Lemma \ref{lemmabounddeltag} that
\begin{align}\label{boundDeltaG1}\int_{\mathfrak{M}(Q_4)}
|F_2^\ast(\alpha)F_3^\ast(\alpha)F(\alpha)\widetilde{\Delta}_{G}(\alpha)|d\alpha \ll \mathcal{F}(0)n^{-1}Q_4^{-40}.\end{align}

One has
$$\widetilde{v}_k(\alpha)-v_k^\ast(\alpha)\ll g(0)(\log R)^{-1},$$
and by Lemma \ref{lemmagauss2},
$$\widetilde{g}_k(\alpha)-g_k^\ast(\alpha)\ll g(0)(\log R)^{-1}q^{-\frac{1}{2}+\varepsilon}.$$
We also have
$$g_k^\ast(\alpha)\ll g(0)q^{-\frac{1}{2}+\varepsilon}.$$
Therefore, on writing
$$\Delta'_G(\alpha)=\prod_{k=5}^{11}\widetilde{g}_k(\alpha)-G^\ast(\alpha),$$
one has
\begin{align}\label{boundDeltaG2}\Delta'_G(\alpha)\ll G(0)(\log R)^{-1}q^{-\frac{7}{2}+\varepsilon}.\end{align}

On applying Lemma \ref{lemmagauss1} and Lemma \ref{lemmaboundvk}, we have
\begin{align}\label{boundF2ast}F_2^\ast(\alpha)\ll F_2(0)q^{-\frac{1}{2}+\varepsilon}(1+n|\beta|)^{-2}.\end{align}
Then it follows from \eqref{boundDeltaG2} and \eqref{boundF2ast} that
\begin{align}\label{boundDeltaG3}\int_{\mathfrak{M}(Q_4)}
|F_2^\ast(\alpha)F_3^\ast(\alpha)F(\alpha)\Delta_{G}'(\alpha)|d\alpha \ll \mathcal{F}(0)n^{-1}(\log n)^{-1}.\end{align}
Note that $\Delta_{G}(\alpha)=\widetilde{\Delta}_{G}(\alpha)+\Delta'_{G}(\alpha)$. Now \eqref{boundDeltaG} follows from \eqref{boundDeltaG1} and \eqref{boundDeltaG3}. The proof is complete.
\end{proof}

\vskip3mm

\section{Proof of Theorem \ref{theorem1}}

Let \begin{align*}\mathfrak{M}=\mathfrak{M}(\log n).\end{align*}
For a measurable set $\mathfrak{B}\subseteq \mathfrak{M}(Q_4)$, we introduce
$$\mathcal{N}_3(n;\mathfrak{B})=\int_{\mathfrak{B}}F_2^\ast(\alpha)F_3^\ast(\alpha)F(\alpha)G^\ast(\alpha)e(-n\alpha)d\alpha.$$

\begin{lemma}\label{lemmaN3}One has
$$\mathcal{N}(n)-\mathcal{N}_3(n;\mathfrak{M}) \ll \mathcal{F}(0)n^{-1}(\log n)^{-1}.$$
\end{lemma}
\begin{proof}On applying Lemma \ref{lemmaN2}, Lemma \ref{lemmaintM3to4} and Lemma \ref{lemmaDeltaG}, we conclude that
\begin{align}\label{boundN31}\mathcal{N}(n)-\mathcal{N}_3(n;\mathfrak{M}(Q_4)) \ll \mathcal{F}(0)n^{-1}(\log n)^{-1}.\end{align}
For $\alpha\in \mathfrak{M}(Q_4)$, one has by Lemma \ref{lemmaboundvk} and Lemma \ref{lemmagauss2} that
$$F_2^\ast(\alpha)G^\ast(\alpha)\ll F_2(0)G(0)(q+qn|\beta|)^{-4+\varepsilon}.$$
Then for $\alpha\in \mathfrak{M}(Q_4)\setminus \mathfrak{M}$, one has
$$F_2^\ast(\alpha)G^\ast(\alpha)\ll F_2(0)G(0)(\log n)^{-1}(q+qn|\beta|)^{-3+\varepsilon},$$
and therefore, by the trivial bound $F_3^\ast(\alpha)F(\alpha) \ll F_3(0)F(0)$, we deduce that
\begin{align}\label{boundintM4toM5}\int_{\mathfrak{M}(Q_4)\setminus\mathfrak{M}}
|F_2^\ast(\alpha)F_3^\ast(\alpha)F(\alpha)G^\ast(\alpha)|d\alpha \ll \mathcal{F}(0)n^{-1}(\log n)^{-1}.\end{align}
The proof is complete by combining \eqref{boundN31} and \eqref{boundintM4toM5}.
\end{proof}

\noindent {\it Proof of Theorem \ref{theorem1}.}
For $\alpha=a/q+\beta\in \mathfrak{M}$ with $1\le a\le q\le \log n$, $(a,q)=1$ and $|\beta|\le (\log n)(qn)^{-1}$, one has
$$f_k(\alpha)=f_k(a/q)+O(Y_k(\log n) n^{\lambda-1}),$$
and by Lemma 5.4 in \cite{Vaughan1989}, there exists  $Y_k'$ satisfying $Y_k\ll Y_k'\ll Y_k$, such that
$$f_k(a/q)-\frac{1}{q}S(q,a)Y_k' \ll Y_k (\log n)^{-1}q.$$
Therefore, we have
\begin{align}\label{boundfk}f_k(\alpha)-\frac{1}{q}S(q,a)Y_k' \ll Y_k (\log n)^{-1}q.\end{align}

On writing
$$f_k^\ast(\alpha)=\frac{1}{q}S(q,a)Y_k'
\ \textrm{ and }\
F^\ast(\alpha)=\prod_{k\in K_2}f_k^\ast(\alpha),$$
one concludes from \eqref{boundfk} that
$$F(\alpha)-F^\ast(\alpha)\ll F(0)(\log n)^{-1}q.$$
Then we deduce that
\begin{align}\label{boundN3toNast}\mathcal{N}_3(n;\mathfrak{M})-\mathcal{N}^\ast(n;\mathfrak{M})\ll \mathcal{F}(0)n^{-1}(\log n)^{-1},\end{align}
where
$$\mathcal{N}^\ast(n;\mathfrak{M})=\int_{\mathfrak{M}}F_2^\ast(\alpha)F_3^\ast(\alpha)F^\ast(\alpha)G^\ast(\alpha)e(-n\alpha)d\alpha.$$

We define
$$A(q)=\frac{1}{q^{6}\phi(q)^7}\sum_{\substack{a=1\\ (a,q)=1}}^{q}\Big(\prod_{k\in K_2\cup\{2,3\}}S_k(q,a)\Big)\Big(\prod_{k\in K_1}S_k^\ast(q,a)\Big)e(-an/q).$$
On applying Lemma \ref{lemmagauss1} and Lemma \ref{lemmagauss2}, one has
\begin{align}\label{boundAq}A(q)\ll q^{-\frac{7}{2}}.\end{align}
Then we introduce
$$\mathfrak{S}(n;X)=\sum_{q=1}^{X}A(q)$$
and write
$\mathfrak{S}(n)=\mathfrak{S}(n;\infty).$
Let
$$v(\beta)=v_2(\beta)v_3(\beta)\prod_{k=5}^{11}v_k^\ast(\beta).$$
Then we define
$$\mathfrak{I}(n;X)=\frac{1}{(\log R)^{r'}}(\prod_{k\in K_1}Y_k')\int_{-X}^{+X}v(\beta)e(-n\beta)d\beta,$$
where $X>0$ and
$$r'=\sum_{k=5}^{11}\frac{\eta^{-1}}{k}.$$
On applying Lemma \ref{lemmaboundvk}, one has
\begin{align}\label{boundIX}\mathfrak{I}(n;X)-\mathfrak{I}(n)\ll \mathcal{F}(0)n^{-1}(1+nX)^{-1},\end{align}
where $\mathfrak{I}(n)=\mathfrak{I}(n;\infty).$
We have
$$\mathcal{N}^\ast(n;\mathfrak{M})=\sum_{q\le \log n}A(q)\mathfrak{I}(n;\frac{\log n}{qn}).$$
Then by \eqref{boundAq} and \eqref{boundIX}, we deduce that
\begin{align}\label{boundNasttoSI}\mathcal{N}^\ast(n;\mathfrak{M})=\mathfrak{S}(n)\mathfrak{I}(n)+O(\mathcal{F}(0)n^{-1}(\log n)^{-1}).\end{align}
We finally conclude from Lemma \ref{lemmaN3}, \eqref{boundN3toNast} and \eqref{boundNasttoSI} that
\begin{align*}\mathcal{N}(n)=\mathfrak{S}(n)\mathfrak{I}(n)+O(\mathcal{F}(0)n^{-1}(\log n)^{-1}).\end{align*}

One can deduce by using the standard argument in the application of the circle method that
$\mathfrak{S}(n)\gg 1$ and $\mathfrak{I}(n)\gg \mathcal{F}(0)n^{-1}$.  In particular, $\mathcal{N}(n)\gg \mathcal{F}(0)n^{-1}$. The proof of Theorem \ref{theorem1} is now complete.

\vskip4mm

\vskip9mm

\begin{thebibliography}{2010}


\bibitem{Br87} J. Br\"udern, {\it Sums of squares and higher powers, II},
J. London Math. Soc.  (2) {\bf 35} (1987), 244--250.


\bibitem{Br88} J. Br\"udern, {\it A problem in additive number theory}, Math. Proc. Cambridge Philos. Soc. {\bf 103} (1988), 27--33.


\bibitem{Br} J. Br\"udern, {\it Ternary additive problems of Waring's type},
Math. Scand  {\bf 68} (1991), 27--45.

\bibitem{BW} J. Br\"udern and T. D. Wooley, {\it On Waring's problem: two cubes and seven biquadrates}, Tsukuba J. Math. {\bf 24}
(2000), 387-417.


\bibitem{Ford1} K. B. Ford, {\it The representation of numbers as sums of unlike powers},
J. London Math. Soc.  (2) {\bf 51} (1995), 14--26.




\bibitem{Ford} K. B. Ford, {\it The representation of numbers as sums of unlike powers. II}, J. Amer. Math. Soc.
 {\bf 7} (1996), 919--940.



\bibitem{KW} K. Kawada and T. D. Wooley, {\it On the Waring-Goldbach
problem for fourth and fifth powers}, Proc. London Math. Soc. (3) {\bf 83} (2001), 1--50.


\bibitem{McD} S. McDonagh, {\it On the sum} $\sum_{t<N^{1/k}}d(N-t^k)$, Proc. Edinburgh Math. Soc. (2) {\bf 15} (1967), 215--219.



\bibitem{Roth} K. F. Roth, {\it A problem in additive number theory}, Proc. London Math. Soc. (2) {\bf 53} (1951), 381--395.

\bibitem{Th1} K. Thanigasalam, {\it On additive number theory}, Acta Arith {\bf 13} (1968), 237--258.

\bibitem{Th2} K. Thanigasalam, {\it On sums of powers and a related problem}, Acta Arith {\bf 36} (1980), 125--141.

\bibitem{Th3} K. Thanigasalam, {\it On certain additive representations of integers}, Portugal. Math. {\bf 42} (1983-1984), 447--465.

\bibitem{Vaughan1970} R. C. Vaughan, {\it On the representation of numbers as sums of powers of natural numbers}, Proc. London Math. Soc. (3) {\bf 21} (1970), 160--180.

\bibitem{Vaughan1971} R. C. Vaughan, {\it On sums of mixed powers}, J. London Math. Soc. (2) {\bf 3} (1971), 677--688.

    \bibitem{Vaughan1986} R. C. Vaughan, {\it On Waring's problem for smaller exponents}, Proc. London Math. Soc. (3) {\bf 52} (1986), 445--463.

   \bibitem{Vaughan1989} R. C. Vaughan, {\it A new iterative method in Waring's problem}, Acta Math. {\bf 162} (1989), 1--71.


\bibitem{V}
R. C. Vaughan, {\it The Hardy-Littlewood method}, 2nd ed.
Cambridge University Press, Cambridge 1997.


    \bibitem{VW} R. C. Vaughan and T. D. Wooley, {\it Further improvements in Waring's problem}, Acta Math. {\bf 174} (1995), 147--240.


    \bibitem{VW2} R. C. Vaughan and T. D. Wooley, {\it Further improvements in Waring's problem, IV: Higher powers}, Acta Arith. {\bf 94} (2000), 203--285.


   \bibitem{Wooley1992} T. D. Wooley, {\it Large improvements in Waring's problem}, Ann. Math. {\bf 132} (1992), 131--164.

    \bibitem{Wooley1995} T. D. Wooley, {\it Breaking classical convexity in Waring's problem: sums of cubes and quasidiagonal
behaviour}, Invent. Math. {\bf 122} (1995), 421--451.
\end{thebibliography}
\end{document}